\documentclass[11pt,a4paper]{article}
\usepackage[margin=30truemm]{geometry}
\usepackage{amsmath,amsthm,amssymb,amsfonts}
\usepackage{enumitem}
\usepackage{color} 
\usepackage{dsfont} 
\usepackage{booktabs,multirow}
\usepackage{url}
\usepackage{authblk}
\usepackage{hyperref}
\usepackage{algorithm}
\usepackage{algpseudocode}

\usepackage{titlesec}
\titleformat*{\section}{\large\bfseries}
\titleformat*{\subsection}{\normalsize\bfseries}
\titleformat*{\subsubsection}{\normalsize\itshape}

\usepackage{mleftright}\mleftright 

\usepackage{graphicx}
\graphicspath{{./fig/}}

\usepackage{dsfont}

\newtheorem{theorem}{Theorem}
\newtheorem{proposition}[theorem]{Proposition}%
\newtheorem{lemma}[theorem]{Lemma}%
\newtheorem{corollary}[theorem]{Corollary}

\newtheorem{example}{Example}%
\newtheorem{remark}{Remark}%

\newtheorem{definition}{Definition}%

\newcommand{\bN}{\mathbb{N}}
\newcommand{\bZ}{\mathbb{Z}}
\newcommand{\bR}{\mathbb{R}}
\newcommand{\bE}{\mathbb{E}}

\newcommand{\bP}{\mathbb{P}}

\newcommand{\cF}{\mathcal{F}}

\newcommand{\variance}{\mathbb{V}ar}

\newcommand{\One}{\mathds{1}}

\newcommand{\one}{\mathds{1}}

\newcommand{\probspace}{(\Omega, \cF, \bP)}

\newcommand{\rf}{\{X_t\,: \,t\in \bR^d\}}
\newcommand{\fre}{\text{Fr\'echet}}
\newcommand{\eX}{\hat{X}}

\DeclareMathOperator*{\argmin}{arg\,min}
\DeclareMathOperator*{\argmax}{arg\,max}

\begin{document}

\title{Extrapolation of max--stable random fields with Fr\'echet marginals}

\author{%
\fbox{Vitalii Makogin}$^{1}$ \quad
Evgeny Spodarev$^{1}$ \quad
Ilja Sukhanov$^{1,*}$
}

\date{%
$^{1}$Institute of Stochastics, Ulm University, Helmholtzstra\ss e 18, Ulm, 89069, Baden-W\"urttemberg, Germany.\\
E-mail: \texttt{evgeny.spodarev@uni-ulm.de}, \texttt{ilja.sukhanov@uni-ulm.de}\\
$^{*}$Corresponding author
}

\maketitle

\begin{abstract}
    We propose a method for the prediction of stationary max--stable random fields with $\alpha$-Fr\'echet marginal distribution $H_\alpha$. The method is suitable to cope with heavy tails for $\alpha\in(0,2)$ and is (approximately) exact in marginal distributions. It is based on a recent extrapolation approach via level sets  which requires no  moment assumptions. An explicit connection between the excursion metric and the Davis-Resnick distance is established. The existence of the predictor is proven.
The non-uniqueness of the forecast is demonstrated on several examples. The method is tested on multiple simulated time series and random fields as well as applied to real data of annual maximum precipitation.
\end{abstract}


\medskip

\section{Introduction}\label{sec1}

Motivated by the paper \cite{spodarev2022} on excursion-based extrapolation of random functions, we apply its ideas to the prediction of max-stable random fields. Precisely, we will use observations of a stationary real valued max-stable random field $\rf$ with Fr\'echet marginals to predict its unobserved values. Here, Fr\'echet distribution was chosen exemplarily to showcase that our method can deal with heavy tails. Likewise,  our methodology can be applied (with few modifications) to stationary max-stable random fields with Gumbel or Weibull univariate probability laws. 

Max-stable distributions have been extensively used for modeling extreme events in e.g. hydrological \cite{papalexiou2013, morrison2002, westra2011}, financial \cite{morales2005, yuen2014, zhang2002} and epidemical context \cite{lim2020, wong2020}. Prediction done by conventional kriging techniques (see e.g. \cite{prudhomme1999,Yinetal18}) may be sometimes inadequate, as their smoothing property can hardly reproduce extreme values, and they do not preserve the law of the random field. Conditional sampling methods as in \cite{wang2011, dombry2013, oesting2014, ribatet2013,MR3644315}, on the other hand, preserve the probability law but are computationally expensive. While the literature on the extrapolation of stationary square integrable random fields is vast (see e.g. \cite{stein1999, wackernagel2003, chiles2012geostatistics}), few attempts to introduce a general framework for the prediction of heavy-tailed random fields are confined to quantile regression \cite{Komunjer,gneiting2011quantiles,komunjer2013quantile} as well as level sets' prediction \cite{spodarev2022, das2022}. In the max-stable context, it is important to mention the prediction of max-linear and max-stable autregressive moving average (MARMA)  processes with Fr\'echet marginals by the Davis-Resnick distance projection on the so-called max-stable spaces \cite{DavisResnick,DavisResnick93}.

A prediction functional from \cite{spodarev2022}, $\hat{X}_\lambda=\hat{X}(\mathbb{T}_f,\lambda)$, depends on a deterministic weight vector $\lambda\in \Lambda\subset \bR^n$ and observations of the stationary random field $X=\rf$ at a finite set of points $\mathbb{T}_f =\{ t_1,\ldots, t_n  \} \subset \mathbb{R}^d$.
The predictor $\hat{X}(\mathbb{T}_f,\hat{\lambda})$ of the unobservable random variable $X_{t_0}$ is the value of $\hat{X}_\lambda$ at the "optimal" point $\hat{\lambda}$, which is defined via minimization procedure  
\begin{equation}
    \label{eq1}
\hat{\lambda}=\arg\min_{\lambda\in \Lambda} \left\{\mathcal{E}_F(\hat{X}_\lambda,X_{t_0})+\gamma \omega_2^2(F,Law(\hat{X}_\lambda))\right\} \,.
\end{equation} 
Here, $F$ is the marginal distribution of $X$ and $\mathcal{E}_F$ is the excursion metric that measures the distance between the predictor $\hat{X}_\lambda$ and the unobsered value $X_{t_0}$ (see Definition \ref{excursion_metric_definition}). In addition, the 2-Wasserstein distance $\omega_2$ controls the difference between their (marginal) distributions. Since in most practical cases (apart from Gaussian, cf. \cite{das2022}) it is not possible to find $\hat{\lambda}$ analytically, the functionals $\mathcal{E}_F$ and $\omega_2$ are substituted by their statistical counterparts.

In the present paper, we apply this framework to stationary max-stable random fields with Fr\'echet--distributed marginals ($F=H_\alpha$) as follows. We require that the predictor random variable $\hat{X}_\lambda$  belongs, for every $\lambda\in \Lambda=\bR_+^n$, to the same family of max-stable distributions as $X_{t_0}$ by setting $$\hat{X}_\lambda=\max_{t_j \in \mathbb{T}_f}\{\lambda_j X_{t_j}\}.$$
Here and throughout the paper, we use the notation $\bN_0=\bN\cup\{0\}$ and $\bR^n_+:=[0,\infty)^n\setminus \{\textbf{0}\}$, where \textbf{0} is the origin of coordinates in $\bR^n$. 

Under the law-preserving max-stable prediction $\hat{X}_\lambda \stackrel{d}{=}X_{t_0}$, the latter condition can be incorporated in Equation \eqref{eq1} by minimization over the corresponding subspace 
$\Lambda_L =\{\lambda\in \Lambda:  \omega_2(H_\alpha,Law(\hat{X}_\lambda))=0  \}. $  
The subspace $\Lambda_L$ is the boundary of a convex subset in $\bR^n_+$ with geometry predefined by the dependence structure of the vector $\{ X_t, \; t\in \mathbb{T}_f\}$ (compare  Remark \ref{remGeomInt}) which makes $\Lambda_L$ hard to use in a computational context. For this reason, we use $\Lambda=\bR^n_+$ in \eqref{eq1} instead of $\Lambda=\Lambda_L$.

We proceed to introduce the max-stable predictor in Section \ref{sec3} following a brief primer on extreme value theory in Section \ref{sec2} and an investigation of metrics $\mathcal{E}_{H_\alpha}$ and $\omega_2$ in Section \ref{subsec23}. There, an easy relation between $\mathcal{E}_{H_1}$ and Davis-Resnick distance $d$ (proposed in \cite{DavisResnick,DavisResnick93}  for the prediction of stationary max--stable processes)  is discussed as well, making the unconstrained metric projections with respect to both metrics equivalent.   Proofs of the existence and non-uniqueness of the predictor are given in Sections \ref{sec4} and \ref{sec5}, respectively. The optimization problem \eqref{eq1} is solved numerically e.g. by a stochastic gradient descent, whose convergence is investigated in Section \ref{sec6}. The methodology is applied to simulated Brown-Resnick, Smith and extremal Gaussian processes  and fields  (introduced in Section \ref{subsec24}), and also tested on real rainfall data in Section \ref{sec7}. A discussion of the results resumes the paper. \textsc{R} software of the methods from Section \ref{sec3} is available at \cite{sukhanov2026code}.

\section{Preliminaries}\label{sec2}
 
In what follows, we assume any random variables and random fields to be defined on a probability space $\probspace$. We recall the two norms $\|\mathbf{x}\|_p:= (x_1^p+\ldots+x_n^p)^{1/p},$ $p\ge 1$, and  $\|\mathbf{x}\|_{\infty}:= \max_{j=1,\ldots,n} \lvert x_j\rvert$, $\mathbf{x}=(x_1,\ldots, x_n)\in\bR^n$.  In general, we write $\mathbf{x}^{\alpha}:=(x_1^{\alpha},\ldots, x_n^{\alpha})$ for any $\mathbf{x}=(x_1,\ldots, x_n)\in \bR^{n}_+$ and $\alpha \in \bR$. Also, we make use of the notation $a\vee b := \max\{a,b\}$, $a\wedge b :=\min\{a,b\}$, $a,b\in\bR$. For any integrable function $f:E\to\bR$, $E\subseteq \bR$, we denote $\| f \|_{L^1(E)}:=\int_E \mid \! \! f(x) \! \! \mid dx $ to be the $L^1$--norm of $f$.  We use the standard notation ${\cal U}[0,1]$ for the uniform probability law
on the interval $[0,1]$ and $C^k(E)$ for the class of $k$ times continuously differentiable functions on a space $E$.

\subsection{Max-stable distributions and random fields}\label{subsec2}

Let $\{X_t:t\in T\}$ be a random function on an arbitrary index set $T$. If for any $m\in\bN$ there exist constants $\alpha_m>0$, $\beta_m\in\bR$, such that the
distribution of 	$\{X_t: t\in T\}$ coincides with the distribution of pointwise maximum $\biggl\{ {\alpha_m}\left(\bigvee_{j = 1}^m X_t^{(j)}- \beta_{m}\right): t\in T\biggr\},$ 
where $\{X_t^{(j)}, t\in T\}$ are i.i.d. copies of $\{X_t: t\in T\}$, then $\{X_t: t\in T\}$ is called a \textit{max-stable random} 	 \textit{variable}, if $\lvert T\rvert =1$;   \textit{vector}, if $\lvert T \rvert = n$, $n\in\bN$;  \textit{process}, if $T=\bR$;  \textit{field}, if $T =\bR^d$, $d\in\bN$, $d\geq 2$. Here $\lvert T\rvert$ denotes the cardinality of a set $T$.

Recall that a random process or field $\rf$ is said to be \textit{stationary} if all its finite dimensional distributions are translation invariant. Following the theorem of Fisher-Tippet-Gnedenko (see e.g. \cite{dehaan2006}), any max-stable random variable (and hence any marginal of a max-stable random field) has one of the three extreme value distributions: Weibull, Gumbel or Fr\'echet.  A random variable $X_0\in\bR_+$ has a univariate $\alpha$-Fr\'echet distribution with minimum 0 and scale  $\sigma >0$ if 
\begin{equation}\label{DefG}
	\bP(X_0\leq x) = H_\alpha(x;\sigma):=\exp\left(-{x}^{-\alpha}\sigma^\alpha\right)\One\{x>0\}, \quad x\in \bR.
\end{equation}
The $\alpha$-Fr\'echet law is the only heavy-tailed one among extreme value distributions and has finite moments of order $\beta<\alpha.$ Using the notation $H_\alpha=H_\alpha(\,\cdot\,;1)$ and  $H=H_1(\,\cdot\,;1)$, it is straightforward that  $H_\alpha(x)=H(x^\alpha),$ $x>0$ and $H_\alpha(\,\cdot\,;\sigma^{1/\alpha})=\left( H_\alpha (\cdot) \right)^\sigma.$ 

Further, an $n$--dimensional max--stable random vector $\mathbf{Y}$  is said to have a {\it  simple $n-$variate max--stable distribution} if all its marginal distribution functions are equal to $H(x)=\exp(-x^{-1})$, $x>0.$ Any $n$--variate max-stable random vector $\mathbf{X}$ 
    with $H_\alpha(\cdot,\sigma)$--distributed marginals can be represented as $ (\sigma Y^{1/\alpha}_1,\ldots, \sigma Y^{1/\alpha}_n)$ with distribution function $G_\alpha(\mathbf {x}):=G(\sigma^{-\alpha} \mathbf {x}^\alpha),$ $\mathbf{x}\in \bR^n_+,$ where $\mathbf{Y}=(Y_1,\ldots,Y_n)$ is a simple $n-$variate max-stable random vector with distribution function $G.$

The multivariate distribution function $G$ of $\mathbf{Y}$ has the following  properties:
\begin{enumerate}
    \item  $G$ has a  uniquely determined exponent measure $\nu$ on $\bR^n_+,$ that is $G(\mathbf{x})=\exp(-\nu([\mathbf{0},\mathbf{x}]^c))$, $\mathbf{x}\in \bR^n$, where $[\mathbf{0},\mathbf{x}]$ is a parallelepiped with left bottom vertex $\mathbf{0}$ and right upper vertex $\mathbf{x}$. 
    \item Max-stability and continuity of $G$ imply that $G(\gamma\mathbf{x})=G^{1/\gamma}(\mathbf{x})$  and $\nu([\mathbf{0},\mathbf{x}]^c])=\gamma\nu([\mathbf{0},\gamma\mathbf{x}]^c])$, $\textbf{x}\in \bR^n_+,$ $\gamma>0$.
    \item The exponent measure $\nu$ factorizes in radial and angular components leading to the {\it De Haan--Resnick representation}
        \begin{equation}
        \label{HR:Repr}
        \nu([\mathbf{0},\mathbf{x}]^c)=\int_{\mathbb{S}_n^+}\bigvee_{j=1}^{n}\frac{q_j}{x_j}\phi(d\mathbf{q}),\quad \mathbf{x}=(x_1,\ldots, x_n)\in \bR_+^n,
    \end{equation}
    where  the angular measure $\phi$ on $\mathbb{S}_n^+:=\{\mathbf{x}\in \bR_+^n:\|\mathbf{x}\|=1\}$ is finite and satisfies $\int_{\mathbb{S}_n}q_j \phi(d\mathbf{q})=1.$ Here $\|\cdot\|$ is an arbitrary norm on $\bR^n_+$.

    \item Alternatively, $\nu$ possesses the spectral representation 
    \begin{equation}
        \label{Sp:Repr}
        \nu([\mathbf{0},\mathbf{x}]^c)=\int_0^1 \bigvee_{j=1}^{n}\left(\frac{p_j(u)}{x_j}\right) du,\quad \mathbf{x}=(x_1,\ldots, x_n)\in \bR_+^n,
    \end{equation} where $p_j$ are probability densities on the interval $[0,1]$, i.e., non-negative functions satisfying $\int_0^1p_j(du)=1,$ $j=1,\ldots,n.$
\end{enumerate}


\subsection{Tail dependence functions, max-linear combinations and D-norms}
From now on, let $\mathbf{X}$ be an $n$-dimensional max-stable random vector with $H_\alpha$-distributed marginals and multivariate distribution function $G_\alpha$.
The exponent $l_{\mathbf{X}}(\mathbf{y}):=-\log G_\alpha(\mathbf{y}^{-1/\alpha}),$ $\mathbf{y}=(y_1,\ldots, y_n)\in \bR_+^n$ is called the { \it tail dependence function} of a random vector $\mathbf{X}$, see \cite{joe1996, joe1990, resnick1987}.  
Here the notation $\mathbf{y}^{-1/\alpha}$ refers to a vector with  coordinates  $(1/y_1^{1/\alpha},\ldots, 1/y_n^{1/\alpha})$. 
The function $l_{\mathbf{X}}$ is convex, homogeneous of order 1 and satisfies
		\begin{equation}\label{eq:Ineql}
        \|\mathbf{y}\|_{\infty} \leq l_{\mathbf{X}}(\mathbf{y})\leq \|\mathbf{y}\|_{1},\quad {\bf y} \in \bR^n_+,\end{equation}
	where the upper bound represents stochastic independence and the lower bound represents perfect positive dependence of the coordinates of $\mathbf{X}$, see \cite{gudendorf2010}. By \eqref{Sp:Repr}, we have \begin{equation}\label{eq:l_w_int}l_{\mathbf{X}}(\mathbf{x}) =\left\| \bigvee_{j=1}^{n}{x_j}{p_j}  \right\|_{L^1[0,1]}, \quad \mathbf{x}=(x_1,\ldots,x_n)\in\bR^n_+.\end{equation}
Moreover, it holds   
$
    l_\mathbf{X}(\mathbf{x})=l_\mathbf{Y}(\mathbf{x}^\alpha),\quad  \mathbf{x}\in\bR_+^n,
$
where $l_\mathbf{Y}$ is the tail dependence function of a simple max-stable random vector $\mathbf{Y}$. 
      

Let $M(\mathbf{w},\mathbf{X}):=\bigvee_{j=1}^{n}{(w_j X_j)}$ be the max-linear combination of  random vector $\mathbf{X}=(X_1,\ldots,X_n)$ with weight vector $\mathbf{w}=(w_1,\ldots, w_n)\in \bR_+^n$. 
We denote by $S_{\mathbf{X}}(\mathbf{w})$ the scale parameter of  $M(\mathbf{w},\mathbf{X})$. Then  
\begin{align}
    &\bP(M(\mathbf{w},\mathbf{X})\leq z)=\bP(X_1\leq z/w_1,\ldots,X_n\leq z/w_n) \label{eqScaleM} \\
    &=G_\alpha( z^\alpha w_1^{-\alpha},\ldots,z^\alpha w^{-\alpha}_n)
    =\exp\left(-z^{-\alpha} l_{\mathbf{X}}(\mathbf{w}^{\alpha})\right)= H_\alpha(z,S(\mathbf{w})),\notag
\end{align}
which means
$
 S_\mathbf{X}(\mathbf{w}) 
=l_{\mathbf{X}}^{1/\alpha}(\mathbf{w}^\alpha).
$

By \cite{Haan78}, the stochastic process $\{X_t, t\in T\}$ with $\alpha$-Fr\'echet marginals is max-stable iff all positive max-linear combinations $ 
    M(\mathbf{w},(X_{t_1},\ldots,X_{t_n}))=\bigvee_{j=1}^{n}w_j X_{t_j}
$ for all $w_j>0,$ $t_j\in T,$ $1\leq j\leq n$ are $\alpha$-Fr\'echet distributed random variables.

 Further properties of $l_{\mathbf{X}}$ involve the use of $D-$norms 
 and extremal coefficients, we refer \cite{falk2019} for related topics. 
 \begin{definition}[D-norm]
	Let $\mathbf{Z}=(Z_1,\ldots,Z_n)$ be a random vector with components  satisfying $Z_j\geq 0$ a.s., $\bE Z_j = 1$, $j = 1,\ldots,n$. Then $\|\mathbf{x}\|_{D(\mathbf{Z})}:= \bE\left(\bigvee_{j=1}^n \lvert x_j \rvert Z_j \right),$  $\mathbf{x}=(x_1,\ldots,x_n)\in\bR^n,$ defines the so-called \textit{D-norm}, and $\mathbf{Z}$ is called a {\it generator} of this $D-$norm.
\end{definition}
Each $D-$norm is monotone, i.e., for $\mathbf{0} \leq \mathbf{x} \leq \mathbf{y}$ componentwise, we have $\|\mathbf{x}\|_{D(\mathbf{Z})}\leq \|\mathbf{y}\|_{D(\mathbf{Z})}.$ Each $\|\cdot\|_p,$ $1\leq p\leq \infty$ is a $D-$norm; sup-norm $\|\cdot\|_\infty$ is the smallest and $\|\cdot\|_1$  is the largest $D-$norm.  

$G_\alpha$ is the distribution function of a $n$-variate max-stable random vector $\mathbf{X}$ with $H_\alpha$-distributed margins iff there exists a generator $\mathbf{Z}$ of a $D$-norm on $\bR_+^n$ such that $G_\alpha(\mathbf{x})=\exp(-\|\mathbf{x}^{-\alpha}\|_{D(\mathbf{Z})}),$ $\mathbf{x}\in \bR_+^n$, cf. \cite[Theorem 2.3.4]{falk2019}.  From the above, one concludes that $l_{\mathbf{X}}(\mathbf{x})=\|\mathbf{x}\|_{D(\mathbf{Z})}$, $\mathbf{x}\in \bR_+^n$. For $\alpha>1$, one may use the generator $\mathbf{Z}=\left( \Gamma(1-1/\alpha) \right)^{-1}\mathbf{X}$ 
to get $\|\mathbf{x}\|_{D(\mathbf{Z})}=l_\mathbf{X}^{1/\alpha}(\mathbf{x}^\alpha)$,  $\mathbf{x}\in \bR_+^n$.


For any non--empty subset $A$ of $\{1,\ldots,n\}$ introduce the {\it extremal coefficient} $\theta^A_{\mathbf{X}}$ of a max--stable random vector $\mathbf{X}=(X_1,\ldots,X_n)$  by $\theta^A_{\mathbf{X}}=\int_{\|q\|_1=1}\bigvee_{j\in A}q_j\phi(d\mathbf{q}),$ where $\phi$ is the angular measure from representation \eqref{HR:Repr} of the distribution function of $\mathbf{X}.$ These coefficients can be written via the corresponding tail dependence function $l_{\mathbf{X}}$ as $\theta^A_{\mathbf{X}}=l_{\mathbf{X}}\left(\sum_{j\in A}\mathbf{e}_j\right)\in [1,2],$ where $\mathbf{e}_j$ is the $j$-th orthonormal basis vector in $\mathbb{R}^n$, $1 \leq j \leq n$. The value $\theta^A_{\mathbf{X}}=1$ corresponds to complete dependence of coordinates of $\mathbf{X}$ with indices in $A$,  whereas $\theta^A_{\mathbf{X}}=2$ means their stochastic independence.

The  copula $C_\mathbf{X}$ of $\mathbf{X}$  
 is defined by the relation
$$C_\mathbf{X}(\mathbf{u}):= G_\alpha(-(\log u_1 )^{-1/\alpha},\ldots,-(\log u_n )^{-1/\alpha}),\quad   \mathbf{u}=(u_1,\ldots,u_n)\in[0,1]^n.$$    Particularly, the copula's diagonal equals \begin{equation*}\label{eqCopulDiag}
    C_\mathbf{X}(u\mathbf{1})=
u^{l_\mathbf{X}(\mathbf{1})}= u^{\theta^{\{1,\ldots,n\}}_{\mathbf{X}}}, \quad  u\in [0,1],\end{equation*} where $\mathbf{1}=(1,1,\ldots, 1)$,   cf. \cite{gudendorf2010}.


\section{Probability metrics}\label{subsec23}

In what follows, we introduce several distances (excursion, Davis--Resnick, Wasserstein) on the spaces of random variables or their distribution functions. Those metrics, tailored to different aspects of probabilistic modeling such as coping with heavy tails or max--stability, will be used to evaluate one random variable's ability to predict or approximate the (distribution of) another random variable. 

\subsection{Excursion metric}

As stated in the introduction, we make a forecast by minimizing the excursion metric from \cite{spodarev2022} between an unobservable random variable and a given predictor. 
\begin{definition}\label{excursion_metric_definition}
	Let $\mu$ be a probability measure on $\bR$ and $Y_1, Y_2$ be two real-valued random variables. The excursion (pseudo-) metric $\mathcal{E}_\mu$ is then defined by $$	\mathcal{E}_\mu (Y_1, Y_2) := \int_\bR \left(\bP(Y_1 \vee Y_2 > u) - \bP(Y_1 \wedge Y_2 > u)\right)\mu(du).$$
\end{definition}
If $\mu$ has a continuous cumulative distribution function $F$, then the above definition rewrites as
		\begin{eqnarray}
		    \label{eq:ExcMetr}
        \mathcal{E}_F (Y_1, Y_2) &= \bE\left(F(Y_1 \vee Y_2) -F(Y_1 \wedge Y_2)\right)\\
        &= 2\bE F(Y_1 \vee Y_2) - \bE F(Y_1) -\bE F(Y_2)  .\nonumber
		\end{eqnarray}
If $F$ is additionally strictly increasing, then $\mathcal{E}_F$ is indeed a metric on the space of absolutely continuous random variables with values in the support of the distribution $F.$ In this case, $\mathcal{E}_F$ is proportional to the $F$--{\it madogram} in geostatistics, cf. e.g. \cite[p. 379]{Cooley2006}. 

It is shown in \cite{spodarev2022} that $0\le \mathcal{E}_F (Y_1, Y_2) \le 1$ with $Y_1=Y_2$ a.s. in case $ \mathcal{E}_F (Y_1, Y_2)=0$. However, if random variables $Y_1$ and $Y_2$ share the same distribution function $F$ (the so-called Gini metric), then the upper bound of  $ \mathcal{E}_F (Y_1, Y_2)$ reduces to $1/2$ which is attained iff $Y_1$ and $Y_2$ are a.s. affine-linear dependent. If $Y_1,$ $Y_2$ are stochastically independent, then $ \mathcal{E}_F (Y_1, Y_2)=1/3$.

We use $\mathcal{E}_F$ with $F=H_\alpha$ to construct a pseudometric $\rho_\mathbf{X}:\bR^n_+\times \bR^n_+\to \bR_+$ between two max-linear combinations of a max-stable random vector $\mathbf{X}$ with $H_\alpha$-distributed margins by 
\begin{align}
\label{rho:def}
    \rho_\mathbf{X}(\mathbf{w}^{(1)},\mathbf{w}^{(2)}):=\mathcal{E}_{H_\alpha} (M(\mathbf{w}^{(1)},\mathbf{X}), M(\mathbf{w}^{(2)},\mathbf{X})) ,\quad \mathbf{w}^{(1)},\mathbf{w}^{(2)}\in\bR_+^n.
\end{align}
Further, we set $\mathbf{w}^{(1)}\vee \mathbf{w}^{(2)}= (w^{(1)}_k\vee w^{(2)}_k, \; k=1,\ldots, n)$ for $\mathbf{w}^{(j)}=(w^{(j)}_1,\ldots,w^{(j)}_n)\in \bR_+^n$, $j=1,2.$
We start with the case $n=2$ and $\alpha=1$.
\begin{proposition}\label{EM:d2}
    Let $(X_1,X_2)$ be a simple max-stable random vector. Then for $\sigma_1,\sigma_2>0$ we have
    \begin{equation}\label{EM:d2:eq}
    \mathcal{E}_H (\sigma_1 X_1, \sigma_2 X_2)=\frac{1}{1+\sigma_1}+\frac{1}{1+\sigma_2}-\frac{2}{1+l_{(X_1,X_2)}((\sigma_1,\sigma_2))}.
    \end{equation}
\end{proposition}
\begin{proof}
    Random variables $H(X_1)$ and $H(X_2)$ are uniformly distributed, i.e. $H(X_1)\stackrel{d}{=}H(X_2)\stackrel{d}{=}U_0\sim {\cal U}[0,1].$ Then 
    $H(\sigma_jX_j)=H^{\sigma^{-1}_j}(X_j) \stackrel{d}{=}U_0^{\sigma^{-1}_j},$ $j=1,2.$ Moreover, $\bP(\sigma_1X_1\leq u, \sigma_2X_2\leq u)=\bP(X_1\leq u \sigma^{-1}_1, X_2\leq u \sigma^{-1}_2)=G(u(\sigma^{-1}_1,\sigma^{-1}_2))=H(u l^{-1}_{(X_1,X_2)}(\sigma_1,\sigma_2)).$ Thus, $$H(\sigma_1X_1 \vee \sigma_2X_2)\stackrel{d}{=} U_0^{l^{-1}_{(X_1,X_2)}(\sigma_1,\sigma_2)} $$ By definition of excursion metric, $\mathcal{E}_H (\sigma_1 X_1, \sigma_2 X_2)= 2\bE H(\sigma_1X_1 \vee \sigma_2 X_2) - \bE H(\sigma_1 X_1) -\bE H(\sigma_2 X_2),$ which leads directly to \eqref{EM:d2:eq}.
\end{proof}
Relation \eqref{EM:d2:eq} can be generalized to $n$ variables and arbitrary $\alpha>0$ as follows:
\begin{theorem}\label{Em:MC:l}
Let $\mathbf{X}=(X_1,\ldots,X_n)$ be a  $n$-variate max-stable random vector    with $H_\alpha$--distributed marginals  and tail dependence function $l_\mathbf{X}.$ For any $\mathbf{w}^{(1)},\mathbf{w}^{(2)}\in\bR_+^n$ it holds
\begin{eqnarray}\label{Em:MC:l:eq}
    \rho_\mathbf{X}(\mathbf{w}^{(1)},\mathbf{w}^{(2)})&=  \frac{1}{1+l_\mathbf{X}\left((\mathbf{w}^{(1)})^\alpha\right)}+\frac{1}{1+l_\mathbf{X}\left((\mathbf{w}^{(2)})^\alpha\right)} \\
    &- \frac{2}{1+l_\mathbf{X}\left((\mathbf{w}^{(1)}\vee \mathbf{w}^{(2)})^\alpha\right)} \nonumber  .
\end{eqnarray}
\end{theorem}
\begin{proof} 
The random variable $M(\mathbf{w},\mathbf{X})$  has scale parameter $l^{1/\alpha}_\mathbf{X}(\mathbf{w}^\alpha)$, therefore $l^{-1/\alpha}_\mathbf{X}(\mathbf{w}^\alpha) M(\mathbf{w},\mathbf{X})$ is a $H_\alpha$-distributed random variable. 
Thus, with $U_0\sim {\cal U}[0,1]$ we have 
$$H_\alpha( M(\mathbf{w},\mathbf{X}))\stackrel{d}{=}U_0^{l^{-1}_\mathbf{X}(\mathbf{w}^\alpha)},\quad\bE H_\alpha( M(\mathbf{w},\mathbf{X})) = \frac{1}{l^{-1}_\mathbf{X}(\mathbf{w}^\alpha)+1}=1-\frac{1}{1+l_\mathbf{X}(\mathbf{w}^\alpha)}.$$
Obviously, $M(\mathbf{w}^{(1)},\mathbf{X}) \vee M(\mathbf{w}^{(2)},\mathbf{X})=\bigvee_{j=1}^n (w^{(1)}_j\vee w^{(2)}_j)X_j=M(\mathbf{w}^{(1)}\vee \mathbf{w}^{(2)},\mathbf{X}).$ Finally, we get formula \eqref{Em:MC:l:eq} from the relation $\mathcal{E}_{H_\alpha}(M(\mathbf{w}^{(1)},\mathbf{X}), M(\mathbf{w}^{(2)},\mathbf{X}))=2\bE H_\alpha\left(M(\mathbf{w}^{(1)},\mathbf{X}) \vee M(\mathbf{w}^{(2)},\mathbf{X})\right) - \bE H_\alpha(M(\mathbf{w}^{(1)},\mathbf{X})) -\bE H_\alpha(M(\mathbf{w}^{(2)},\mathbf{X})).$
\end{proof}


\begin{corollary}\label{cor:exMetr} Let  the random vector $\mathbf{X}$ be as in Theorem \ref{Em:MC:l}.
    For every vector $\mathbf{w}=(w_1,\ldots,w_n)\in \bR_+^n$ and $\gamma_k\in[0,w_k]$,  $k=1,\ldots,n$,   
    it holds
    \begin{equation}\label{Em:MC2:eq}
 \mathcal{E}_{H_\alpha} (M(\mathbf{w},\mathbf{X}), \gamma_k X_k)= \frac{1}{\gamma_k^\alpha +1}-\frac{1}{l_\mathbf{X}(\mathbf{w}^\alpha)+1}.
\end{equation}
 Let the extended random vector $\mathbf{X}_e=(X_0,\mathbf{X})$ be max--stable   with $H_\alpha$--distributed marginals and tail dependence function $l_{\mathbf{X}_e}$. For $\gamma_0\geq 0$, we have
    \begin{align}\label{Em:MC3:eq}
    \mathcal{E}_{H_\alpha} (M(\mathbf{w},\mathbf{X}), \gamma_0 X_0)= \frac{1}{\gamma_0^\alpha+1} +\frac{1}{l_\mathbf{X}(\mathbf{w}^\alpha)+1}- \frac{2}{l_{\mathbf{X}_e}((\gamma_0^\alpha,\mathbf{w}^\alpha))+1}.
    \end{align}
\end{corollary}
\begin{proof}
Relation   \eqref{Em:MC2:eq} follows from Theorem  \ref{Em:MC:l} if we notice that    $\gamma_k  X_k=  M(\gamma_k\mathbf{e}_k,\mathbf{X})$,  $\mathbf{w}\vee \gamma_k  \mathbf{e}_k = \mathbf{w}$, $k=1,\ldots,n$  for $0\leq \gamma_k\leq w_k$, and  $\bE H_\alpha(M(\gamma_k \mathbf{e}_k,\mathbf{X}))=1-(1+\gamma_k^\alpha)^{-1}.$
To prove relation \eqref{Em:MC3:eq}, we apply Theorem  \ref{Em:MC:l} to the vector  $\mathbf{X}_e$ with $\mathbf{w}^{(1)}=(\gamma_0, 0,\ldots, 0 ) $, $\mathbf{w}^{(2)}=(0, \mathbf{w} ) $.
 \end{proof}

For max-linear combinations,  relation \eqref{eq:l_w_int}
 is instrumental in proving the following result:
\begin{theorem} \label{thm:rho}
Let a max-stable random vector $\mathbf{X}=(X_1,\ldots, X_n)$ with $H_\alpha$--distributed marginals have spectral representation \eqref{Sp:Repr} with spectral functions $p_1,\ldots,p_n>0$ a.e. on $[0,1]$. If functions $p_1,\ldots, p_n$ are max-linearly independent a.e. on $[0,1]$, then the mapping $\rho_\mathbf{X}$ defined in \eqref{rho:def} is a metric on $\bR_+^n\times \bR_+^n.$ 
\end{theorem}
\begin{proof} Let for some $\mathbf{u},\mathbf{v}\in\bR_+^n$ $\rho_{\mathbf{X}}(\mathbf{u},\mathbf{v})=0.$ 
    It is evident by construction that $l_\mathbf{X}(\mathbf{u}^\alpha\vee \mathbf{v}^\alpha)\geq l_\mathbf{X}(\mathbf{u}^\alpha)\vee l_\mathbf{X}(\mathbf{v}^\alpha).$ Moreover, the function $r(x)=1/(x^{-1}+1)$, $x\geq 0$ is monotonically increasing.  Therefore, 
    \begin{align*}
        0=\rho_\mathbf{X}(\mathbf{u},\mathbf{v})&=\underbrace{r(l_\mathbf{X}(\mathbf{u}^\alpha\vee \mathbf{v}^\alpha))-r(l_\mathbf{X}( \mathbf{u}^\alpha))}_{\geq 0}+\underbrace{r(l_\mathbf{X}(\mathbf{u}^\alpha\vee \mathbf{v}^\alpha))-r(l_\mathbf{X}( \mathbf{v}^\alpha))}_{\geq 0}
    \end{align*} implies that $l_\mathbf{X}(\mathbf{u}^\alpha\vee \mathbf{v}^\alpha)= l_\mathbf{X}(\mathbf{u}^\alpha)=l_\mathbf{X}(\mathbf{v}^\alpha).$ 
    Hence, we may write
        \begin{align*}
        0&=2l_\mathbf{X}(\mathbf{u}^\alpha\vee \mathbf{v}^\alpha)-l_\mathbf{X}(\mathbf{u}^\alpha)-l_\mathbf{X}( \mathbf{v}^\alpha)\\
        &=\int_0^1 2\left(\bigvee_{j=1}^n u_j^\alpha p_j(u)\vee \bigvee_{j=1}^n v_j^\alpha p_j(u)\right)du\\
        &-\int_0^1\bigvee_{j=1}^n \left(u_j^\alpha p_j(u)\right)du-\int_0^1\bigvee_{j=1}^n \left(v_j^\alpha p_j(u)\right)du\\
          &=\int_0^1\Bigl\vert \bigvee_{j=1}^n u_j^\alpha p_j(u) - \bigvee_{j=1}^n v_j^\alpha p_j(u) \Bigr\vert du,
    \end{align*}
    which means $\bigvee_{j=1}^n u_j^\alpha p_j(u) = \bigvee_{j=1}^n v_j^\alpha p_j(u)$ for almost all $u\in [0,1]$. This holds under the condition of max-linear independence of positive probability densities $p_j$, $j=1,\ldots, n$ iff $\mathbf{u}=\mathbf{v}$.
    \end{proof}


 \subsection{Davis--Resnick distance \texorpdfstring{$d$}{d}}\label{Subsect:DR}

Let us discuss the relationship between the excursion metric $ \mathcal{E}_{H}$ and a specific dependency distance $d$ for max-stable random variables introduced in \cite{DavisResnick}.   For  a simple max-stable random vector $(X_1,X_2)$  and $\sigma_1,\sigma_2 \geq 0$,  define
\begin{equation*}
\label{Md:eq}
    d(\sigma_1 X_1,\sigma_2 X_2):=\|\sigma_1 p_1-\sigma_2 p_2\|_{L^1([0,1])}=2\|\sigma_1 p_1\vee \sigma_2 p_2\|_{L^1([0,1])}-\sigma_1-\sigma_2, 
\end{equation*}
where functions $p_1,p_2$ come from spectral representation \eqref{Sp:Repr}, i.e., $$\bP(X_1\leq x_1,X_2\leq x_2)=\exp\left(-\int_0^1\frac{p_1(u)}{x_1}\vee \frac{p_2(u)}{x_2}du\right),\quad (x_1,x_2)\in \bR^2_+.$$
Random variables $X_1,X_2$ are completely dependent, i.e., $X_1=X_2$ a.s. iff $d(X_1,X_2)=0.$ They are independent iff $d(X_1,X_2)=2,$ see \cite[Section 3]{DavisResnick}.

\begin{proposition}\label{prop:Dist_d}
Let $(X_1,X_2)$ be a simple max-stable random vector. Then it holds 
$$
\mathcal{E}_{H}(X_1,X_2)=\frac{d(X_1,X_2)}{4+d(X_1,X_2)}, \qquad  d(X_1,X_2)=4\frac{\mathcal{E}_{H}(X_1,X_2)}{1-\mathcal{E}_{H}(X_1,X_2)},
$$
where $\mathcal{E}_{H}$ is the excursion metric and $d$ is the Davis--Resnick distance introduced above.
\end{proposition}
\begin{proof}
Note that $d(\sigma_1 X_1,\sigma_2 X_2)=2 l_{(X_1,X_2)}(\sigma_1,\sigma_2)-\sigma_1-\sigma_2,$ therefore, using relation \eqref{EM:d2:eq} we get 
\begin{align*}
    \mathcal{E}_{H}(X_1,X_2)&=1-\frac{2}{1+l_{(X_1,X_2)}(1,1)}=\frac{d(X_1,X_2)}{4+d(X_1,X_2)},\\
    d(X_1,X_2)&=4\frac{\mathcal{E}_{H}(X_1,X_2)}{1-\mathcal{E}_{H}(X_1,X_2)}.
\end{align*}
\end{proof}
It follows from Proposition \ref{prop:Dist_d} that $\mathcal{E}_{H}(X_1,X_2)=1/3$ iff random variables $X_1$ and $X_2$ are independent. In addition, it holds $ \mathcal{E}_{H}(X_1,X_2)=f(d(X_1,X_2))$, where the function $f(x)=x/(x+4)$ is increasing on $[0,+\infty)$ as well as its inverse $f^{-1}$ is increasing on $[0,1)$. This makes a metric projection with respect to metric $\mathcal{E}_{H}$ or $d$ an equivalent task.

\subsection{Wasserstein distance}

Let $\mathcal{C}_D(\bR)$ be the space of continuous distribution functions on $\bR$. A widely used measure of difference between distributions is  \textit{p-Wasserstein distance}, which is equal to $\omega_p(F_1,F_2)=\|F_1^{-1}-F_2^{-1}\|_{L^p[0,1]}$ for $F_1,F_2\in\mathcal{C}_D(\bR)$, with $p\geq 1$ and quantile functions $F_1^{-1}$ and $F_2^{-1}$, respectively.
Computing the 1- and 2-Wasserstein distances of a probability distribution $F$ to the uniform law ${\cal U}[0,1]$ yields
\begin{align}
    \omega_1(F,{\cal U}[0,1])&=\int_0^1 \lvert F^{-1}(q)-q\rvert dq =\int_{\bR} \lvert y-F(y) \rvert d F(y)=\bE \lvert Y-F(Y)\rvert, \nonumber  \\
    \omega_2^2(F,{\cal U}[0,1])& = \frac{1}{3} - \bE(Y \vee Y_1) + \bE(Y^2) = \frac{1}{3} + \int\limits_0^1 F(u) \left(F(u)- 2u\right) du, \label{eq:2Wasserst}
\end{align}
where $Y,Y_1$ are independent random variables with c.d.f. $F$, cf. \cite{spodarev2022}. 

For max-linear combinations, we will measure the distance between the uniform law ${\cal U}[0,1]$ and the distribution $F_{\mathbf{w}}$ of random variable $H_\alpha(M(\mathbf{w},\mathbf{X}))$ appearing in the excursion metric $\mathcal{E}_{H_\alpha}.$ Recall that $$F_{\mathbf{w}}(u)=\bP(H_\alpha( M(\mathbf{w},\mathbf{X}))\leq u)=\bP\left(U_0\leq u^{l_\mathbf{X}(\mathbf{w}^{\alpha})}\right)=u^{l_\mathbf{X}(\mathbf{w}^{\alpha})},\,u\in [0,1].$$ Direct calculations yield
\begin{equation}\label{eq:1Wasserst}
\omega_1(F_{\mathbf{w}},{\cal U}[0,1])=\int_0^1 \Big\lvert q^{ l^{-1}_\mathbf{X}(\mathbf{w}^{\alpha})   }-q \Big\rvert dq =\frac{1}{2}\frac{\lvert l_\mathbf{X}(\mathbf{w}^{\alpha}) -1\rvert }{l_\mathbf{X}(\mathbf{w}^{\alpha}) +1},\end{equation}
which coincides, by Corollary \ref{cor:exMetr}, with 
$$
\mathcal{E}_{H_\alpha}\left(S_\mathbf{X}(\mathbf{w}) X_k,X_k\right)=
\left\lvert  \frac{1}{2}- \frac{1 }{l_\mathbf{X}(\mathbf{w}^{\alpha}) +1} \right\rvert.
$$
Alongside with relation \eqref{eq:ExcMetr}, this allows for a representation of the $\omega_1$ -metric in terms of max-linear combinations:
$$\omega_1(F_{\mathbf{w}},{\cal U}[0,1])=\frac{1}{2}\lvert 2 \bE H_\alpha( M(\mathbf{w},\mathbf{X}))-1\rvert=\mathcal{E}_{H_\alpha}\left(S_\mathbf{X}(\mathbf{w}) X_k,X_k\right).$$
Similarly, one can write
   \begin{equation} \label{omMS3} 
   \omega_2^2(F_{\mathbf{w}},{\cal U}[0,1]) =\frac{2}{3}\frac{(l_\mathbf{X}(\mathbf{w}^{\alpha})-1)^2}{(2l_\mathbf{X}(\mathbf{w}^{\alpha})+1)(l_\mathbf{X}(\mathbf{w}^{\alpha})+2)}.
\end{equation}

\section{Models of max-stable random fields}\label{subsec24}
Now let us recall three max-stable random field models that we will use for simulation study in Section \ref{sec7}. Denote by $F_\Sigma$ be the cumulative distribution function and by $f_\Sigma$  the probability density function of a $d$-variate centered normal distribution $N(\mathbf{0},\Sigma)$ with positive definite covariance matrix $\Sigma\in \bR^{d\times d}$, i.e.
$	f_\Sigma(x) = (2\pi)^{-d/2} (\text{det}(\Sigma))^{-d/2} \exp\bigl(-\frac{1}{2} x^\top \Sigma^{-1} x\bigr),$ $x\in\bR^d.$

\subsection{Brown-Resnick random field}
 Let $\{Y^{(j)}, j\in\bN\}$ be independent copies of a centered Gaussian random field $\{Y_t, t\in\bR^d\}$ with stationary increments and variance $\variance Y_t = \sigma^2(t)$. Let	$\{\xi_i\}_{j\in\bN}$	be the points of a Poisson process on $\bR$ with intensity measure $e^{-x}dx$ which is independent of $\{Y^{(j)}, j\in\bN\}$. Then the \textit{Brown-Resnick} random field is defined as
	\begin{equation}\label{brownresnick_proc_eq}
		R_t := \max\limits_{j\in\bN}\bigl\{\xi_j + Y^{(j)}_t-\sigma^2(t)/2\bigr\}, \hspace{0.2cm} t\in\bR^d. 
	\end{equation}

It is stationary and has standard Gumbel-distributed margins 
$ 	\bP(R_t\leq x) = e^{-e^{-x}},$ $x\in\bR $  (see \cite{kabluchko2009}). Applying the transformation
\begin{equation}\label{brownresnick_frechet_version}
	B_t := e^{R_t} = \max\limits_{j\in\bN} \exp\bigl(\xi_j+Y^{(j)}_t-\sigma^2(t)/2\bigr), \hspace{0.2cm} t\in\bR^d,
\end{equation} 
we obtain a random field $B_t$ with $H_1$-distributed margins.
The finite-dimensional distributions of the Brown-Resnick random field are given by \cite{kabluchko2009} as
\begin{equation}
   -\log \label{BR:fdd}\bP(R_{t_1}\leq y_1,\ldots,R_{t_n}\leq y_n)=\bE \exp\left\{\bigvee_{j=1}^{n}\left(Y_{t_j}-\frac{\sigma^2(t_j)}{2}-y_j\right)\right\}
\end{equation}
for $t_1,\ldots,t_n\in \bR^d$ and $y_1,\ldots,y_n\in \bR.$ Therefore, the tail dependence function of the max-stable random vector $\mathbf{B}_T=(B_{t_1},\ldots, B_{t_n})$, $T:=(t_1,\ldots,t_n)$ equals 
\begin{align*}
    &l_{\mathbf{B}_T}(x_1,\ldots,x_n)=-\log \bP(B_{t_1}\leq x^{-1}_1,\ldots,B_{t_n}\leq x^{-1}_n)\\
    &=-\log \bP(R_{t_1}\leq -\log x_1,\ldots,R_{t_n}\leq -\log x_n)\\
    &=\bE \exp\left\{\bigvee_{j=1}^{n}\left(Y_{t_j}-\frac{\sigma^2(t_j)}{2}+\log x_j\right)\right\}\\
     &=\bE \bigvee_{j=1}^{n}x_j\exp\left\{\left(Y_{t_j}-\frac{\sigma^2(t_j)}{2}\right)\right\}=\bE \bigvee_{j=1}^{n}x_j  Z_j=\|\mathbf{x}\|_{D(\mathbf{B}_T)}
\end{align*}
for $\mathbf{x}=(x_1,\ldots,x_n)\in\bR_+^n$,
where the vector $\mathbf{Z}=(Z_1,\ldots,Z_n)$ is the generator of the corresponding $D(\mathbf{B}_T)-$norm. Evidently, $Z_j=\exp(Y_{t_j}-\sigma^2(t_j)/2)$ with $\bE Z_j=1,$ $j=1,\ldots, n$ and $\|\cdot\|_{D(\mathbf{B}_T)}$ belongs to the class of so-called \emph{H\"usler-Reis $D-$norms}.
We arrived at the following 
\begin{proposition}
    The tail dependence function $l_{\mathbf{B}_T}$ of $\mathbf{B}_T=(B_{t_1},\ldots, B_{t_n})$  equals
    \begin{align*}
        l_{\mathbf{B}_T}((x_1,\ldots,x_n))&=\int_{\bR^n}\bigvee_{j=1}^n x_j e^{z_j-\frac{1}{2}\sigma^2(t_j)}f_{\Sigma_n}(\mathbf{z})d \mathbf{z}\\
        &=\int_{0}^{+\infty}\left[1-F_{\Sigma_n}\left(\log\frac{y}{x_1}+\frac{\sigma^2(t_1)}{2},\ldots,\log\frac{y}{x_n}+\frac{\sigma^2(t_n)}{2}\right)\right]dy
    \end{align*}
   for $\mathbf{x}=(x_1,\ldots,x_n)\in\bR_+^n$,  where  $\mathbf{z} = (z_1,\ldots,z_n)$, $d\mathbf{z} = dz_1\ldots dz_n$, and $\Sigma_n$ is the covariance matrix of the underlying Gaussian random vector $(Y_{t_1},\ldots,Y_{t_n}).$
\end{proposition}
\begin{proof} The first assertion in the proposition is trivial. As for the second assertion,
it holds $\bE \bigvee_{j=1}^{n}x_j  Z_j=\int_0^\infty (1-F_M(y))dy,$
where $F_M$ denotes the cumulative distribution function of $M:=\exp\left[\bigvee_{j=1}^n\left(Y_{t_j}-\frac{\sigma^2(t_j)}{2}+\log x_j\right)\right].$
 Then 
\begin{align*}
    F_M(y)&=\bP\left(Y_{t_j}-\frac{\sigma^2(t_j)}{2}+\log x_j\leq \log y,\,  j=1,\ldots,n\right)\\
    &=\int_{-\infty}^{\log\frac{y}{x_1}+\frac{\sigma^2(t_1)}{2}}\cdots\int_{-\infty}^{\log\frac{y}{x_n}+\frac{\sigma^2(t_n)}{2}} f_{\Sigma_n}(\mathbf{z})d\mathbf{z}\\
    &=F_{\Sigma_n}\left(\log\frac{y}{x_1}+\frac{\sigma^2(t_1)}{2},\ldots,\log\frac{y}{x_n}+\frac{\sigma^2(t_n)}{2}\right), \quad y>0,
\end{align*}
which completely proves the proposition.
\end{proof}
\begin{example}\label{ex:BRP} Let  $\{B_t:t\in\bR\}$ be the Brown-Resnick process \eqref{brownresnick_frechet_version} with $d=1$ and its underlying process $Y=\{Y_t:t\in\bR\}$ being a Brownian motion with variance $\sigma^2(t)=\sigma^2_Bt$ for some volatility parameter $\sigma_B>0$. The variogram $\gamma(t_1,t_2):=\frac12 \bE (Y_{t_1}-Y_{t_2})^2$ of $Y$ is given as $\gamma(t_1,t_2)=\frac12 \sigma_B^2\lvert t_1-t_2\rvert$. For $n=2$, consider $T=(t_1, t_2)$, $0\leq t_1<t_2$. Then the tail dependence function of the random vector $(B_{t_1}, B_{t_2})$ is 
	\begin{equation}\notag
		 l_{\mathbf{B}_T}((x_1,x_2)) = x_1 F_0\biggl(\sqrt{\frac{\gamma(t_1,t_2)}{2}}+\frac{\log(x_1/x_2)}{\sqrt{2\gamma(t_1,t_2)}}\biggr)+x_2 F_0\biggl(\sqrt{\frac{\gamma(t_1,t_2)}{2}}+\frac{\log(x_2/x_1)}{\sqrt{2\gamma(t_1,t_2)}}\biggr),
	\end{equation}
	where $F_0$ is the standard normal distribution on $\bR$ and $x_1,x_2>0$, see \cite[p.58]{falk2019}.
\end{example}
As shown in \cite[Proposition 3.1]{KablSchlather10}, Brown-Resnick processes ($d=1$) are ergodic whenever $\sigma^2(t)\to +\infty$ as $ t \to +
\infty$.

\subsection{Smith random field} 
Let $\{\xi_j, \varepsilon_j\}_{j\in\bN}$ be the points of a Poisson process on $(0,\infty)\times \bR^d$ with intensity measure $x^{-2}dx dy$. 
Then the \textit{Smith model} is described by the random field
\begin{equation}\label{smith_proc_eq}
	S_t := \max\limits_{j\in\bN} \bigl\{\xi_j f_\Sigma(t-\varepsilon_j)\bigr\}, \hspace{0.2cm} t\in\bR^d,
\end{equation}
where $\Sigma$ is a positive definite $d\times d$ matrix, 
see \cite{smith1990}. 
The finite-dimensional distributions of the Smith model are given by 
\begin{equation*}
    \label{SM:fdd}\bP(S_{t_1}\leq y_1,\ldots,S_{t_n}\leq y_n)=\exp\left(-\int_{\bR^d} \max_{j=1,\ldots,n}\frac{f_\Sigma(t_j-\mathbf{z})}{y_j} d\mathbf{z}\right)
\end{equation*}
for $y_1,\ldots, y_n>0$ and  $t_1,\ldots,t_n\in \bR^d.$ Therefore,   the tail dependence function of  $\mathbf{S}_T=(S_{t_1},\ldots,S_{t_n})$ is given by
    \begin{align*}
        &l_{\mathbf{S}_T}((x_1,\ldots,x_n))=\int_{\bR^d} \bigvee_{j=1}^{n}x_j f_\Sigma(t_j-\mathbf{z})d\mathbf{z}=\int_{\bR^d} \bigvee_{j=1}^{n}x_j \frac{f_\Sigma(t_j-\mathbf{z})}{f_\Sigma(\mathbf{z})}f_\Sigma(\mathbf{z})d\mathbf{z}\\
        &=\int_{\bR^d} \bigvee_{j=1}^{n}x_j \exp\left(t_j^{\top}\Sigma^{-1}\mathbf{z}-\frac{1}{2}t_j^{\top}\Sigma^{-1}t_j \right)f_\Sigma(\mathbf{z})d\mathbf{z}\\
        &=\bE \bigvee_{j=1}^{n}x_j  Z_j=\|\mathbf{x}\|_{D(\mathbf{z}_T)}, \quad \mathbf{x}=(x_1,\ldots, x_n)\in \bR^n_+,
    \end{align*}
    where $Z_j=\exp\left(t_j^{\top}\Sigma^{-1}(\mathbf{A}-t_j/2) \right)=\exp\left(Y_j-\sigma^2(t_j)/2 \right),$ $Y_j:=t_j^{\top}\Sigma^{-1}\mathbf{A},$ 
    $\sigma^2(t):=t^{\top}\Sigma^{-1}t,$  $j=1,\ldots,n$ and $\mathbf{A}\sim N(0,\Sigma).$ Therefore, $D(\mathbf{S}_T)$ is a  H\"usler-Reis $D-$norm as well. Thus, $\{S_t, t\in\bR^d   \}$ can be considered as a special case of Brown-Resnick random fields generated by a linear Gaussian field $Y_t=t^{\top} \tilde{Y},t\in \bR^d$ and $\tilde{Y}=\Sigma^{-1}\mathbf{A}\sim N(0,\Sigma^{-1}),$  and  hence is stationary as well.

    For $d=1$, $n=2$, $T=(t_1,t_2)$, $0\le t_1<t_2$ the tail dependence function of the random vector $(S_{t_1},S_{t_2})$ has the same form as $l_{\mathbf{B}_T}$ from Example \ref{ex:BRP} but with variogram $\gamma(t_1,t_2)=\sigma_S^{-2} (t_1-t_2)^2/2$, where $\sigma_S^2>0$ is the variance parameter of the univariate normal probability density function.

    
 Since $\sigma^2(t)\to +\infty$ as $ t \to +
\infty$, the Smith random process ($d=1$) is ergodic.

    
\subsection{Extremal Gaussian random field}

Let $\{Y^{(j)}, j\in\bN\}$ be independent copies of a stationary centered Gaussian random field $\{Y_t: t\in\bR^d\}$. Let $\{\xi_j\}_{j\in\bN}$ be the points of a Poisson process on $(0,\infty)$ with intensity measure $\sqrt{2\pi} x^{-2}dx$. Then the random field defined by
\begin{equation}\label{extremalgaussian_proc_eq}
	G_t := \max\limits_{j\in\bN} \xi_j (Y^{(j)}_t \vee 0), \hspace{0.2cm} t\in\bR^d 
\end{equation}
is called an \textit{extremal Gaussian} random field.

The extremal Gaussian random field is stationary and has unit Fr\'echet margins, see e.g. \cite[Theorems 1,2]{schlather2002}. The finite-dimensional distributions of the extremal Gaussian random field can be found as follows. Denote by $\Sigma_n$ the covariance matrix of $Y_{t_1},\ldots,Y_{t_n},$ then
$\{(\xi_j, Y^{(j)}_{t_1},\ldots,Y^{(j)}_{t_n}),j\in \bN \}$ equals in distribution to a Poisson point process on $(0,\infty)\times\bR^n$ with intensity measure $$\Lambda(I\times (-\infty,y_1]\times\cdots\times (-\infty,y_n])=\sqrt{2\pi}F_{\Sigma_n}(y_1,\ldots,y_n)\int_I x^{-2}dx$$ for any interval $I\subset (0,\infty)$ and any $y_1,\ldots,y_n\in\bR$.\\
Thus, $\bP(G_{t_1}\leq y_1,\ldots,G_{t_n}\leq y_n)$ equals $\exp(-\Lambda(U)),$
where 
$$U=\{(x,z_1,\ldots,z_n):\,x (z_j \vee 0) >y_j \text{ for some } j=1,\ldots,n\}.$$ Then
\begin{align*}
    \Lambda(U)&=\int_{\bR^n} \int_0^\infty \mathbf{1} \left\{x>\min_{j=1,\ldots,n}\frac{y_j}{z_j\vee 0}\right\} x^{-2}f_{\Sigma_n}(\mathbf{z})dx d\mathbf{z}\\
    &=\int_{\bR^n} \max_{j=1,\ldots,n}\left\{\frac{z_j\vee 0}{y_j}\right\} f_{\Sigma_n}(\mathbf{z})d\mathbf{z}.
\end{align*}
Therefore, for $y_1,\ldots,y_n>0$ we get
\begin{align*}
    &\bP(G_{t_1}\leq y_1,\ldots,G_{t_n}\leq y_n)=\exp\left(-\bE\left[\max_{j=1,\ldots,n}\left\{\frac{Y_{t_j}\vee 0}{y_j}\right\}\right]\right)
    \end{align*}
    and, consequently,  for $\mathbf{G}_T=(G_{t_1},\ldots, G_{t_n})$, $T:=(t_1,\ldots,t_n)$ 
\begin{align*}    
    &l_{\mathbf{G}_T}(\mathbf{x})=\bE\left[0\vee\max_{j=1,\ldots,n}\left\{x_j Y_{t_j}\right\}\right]=\int_0^{+\infty}\left[1-F_{\Sigma_n}\left(\frac{y}{x_1},\ldots,\frac{y}{x_n}\right)\right]dy,
\end{align*}
where $\mathbf{x}=(x_1,\ldots,x_n)$, $x_1,\ldots,x_n\ge 0$. For $n=2$, this  simplifies to 
\begin{align}    \label{eq:L_G}
    &l_{(G_{t_1}, G_{t_2})}((x_1,x_2))=\frac{1}{2} \left(x_1+x_2+\sqrt{x_1^2-2\rho_{t_1 t_2}x_1 x_2+x_2^2}\right),
\end{align}
where $\rho_{t_1 t_2}$ is the correlation coefficient of $Y_{t_1}$ and $Y_{t_2}$, cf. \cite{schlather2002}.

The extremal Gaussian random process ($d=1$) is not ergodic if the correlation function $\rho_{0t}$, $t\in\bR$ of $\{Y_t: t\in\bR\}$ satisfies $\rho_{0t}>-1+\delta$ for some small $\delta>0$ a.e. on $\bR$. To see this, we use \cite[Theorem 3.2]{KablSchlather10} to show that  $\lim_{ T \to +\infty} \frac{1}{T} \int_0^T (2-\theta_G(t))\, dt > 0$, where $\theta_G(t):=l_{(G_0,G_t)}((1,1))$ is the extremal coefficient. Indeed, formula \eqref{eq:L_G} yields $2-\theta_G(t)=1- \sqrt{(1-\rho_{0t})/2}\ge \delta/4 + o(\delta)>0 $ for almost every $t\in \bR$.


\section{Max-stable prediction via excursion sets}\label{sec3}

In this section, we use the prediction theory for heavy-tailed random fields based on their excursion metric \cite{spodarev2022} to make forecasts in the max-stable case. Let $X=\rf$ be a real-valued, stationary and max-stable random field with Fr\'echet($\alpha$) marginals. Let $\bZ_h = (h_1\bZ)\times\ldots\times(h_d\bZ)$ be a $d$-dimensional grid with mesh sizes $(h_1,\ldots,h_d)\in (0,\infty)^d$. The observations of $X$ at points $\mathbb{T}_0:= \bZ_h\cap W_o$ form a sample $ \{X_{t}:t\in \mathbb{T}_0\},$ where $W_o\subset \bR^d$ is a compact set.

Let $t_0\in \bZ_h\setminus \mathbb{T}_0$ be a prediction location. Our goal is to predict the unobserved value $X_{t_0}$ based on the values $X_{t_j}$ at index locations $\mathbb{T}_f:=\{t_1,\ldots,t_n\}\subset \mathbb{T}_0 $ 
via max-stable predictor $$\hat{X}_{\mathbf{\lambda}}:=M(\mathbf{\lambda},\mathbf{X}(\mathbb{T}_f))=\max\limits_{j=1,\ldots,n} \lambda_j X_{t_j},$$ where  $\mathbf{X}(\mathbb{T}_f)=\{ X_{t_j},t_j\in \mathbb{T}_f \}$ is the \textit{forecast sample}. Note that $\mathbb{T}_f$  usually contains $n$ spots $t_j$ closest  to location $t_0.$ The observations  $\{ X_{t_j},t_j\in \mathbb{T}_0 \setminus \mathbb{T}_f\}$ will be used to create multiple learning samples in order to assess the values of the involved excursion and Wasserstein metrics, compare relation \eqref{min_problem} at the end of this section.

\begin{definition}\label{Xp_definition}
	The \textit{law-preserving predictor} of $X_{t_0}$ is given by
	\begin{equation}\notag
		\hat{X}_{\hat{\mathbf{\lambda}}} := \max\limits_{j=1,\ldots,n} \hat{\lambda}_j X_{t_j},
	\end{equation}
	where
	\begin{equation}\label{constraint1}
		\hat{\mathbf{\lambda}} = (\hat{\lambda}_1,\ldots,\hat{\lambda}_n) =\argmin\limits_{\mathbf{\lambda}\in\bR^n_+} \left\{\mathcal{E}_{H_\alpha}\left (X_{t_0},\hat{X}_{{\mathbf{\lambda}}}\right): X_{t_0} \stackrel{d}{=} \hat{X}_{{\mathbf{\lambda}}}\right\}.
	\end{equation}
\end{definition}

The predictor is designed in the form of a max-linear combination in order to preserve the same marginal distribution as the random variable being predicted. The law-preserving constraint in \eqref{constraint1} requires the knowledge of a subset $\Lambda_L\subset\bR^n_+$ such that
$ X_{t_0} \stackrel{d}{=} \hat{X}_\mathbf{\lambda}$,  $\mathbf{\lambda} \in \Lambda_L.$ Since both $X_{t_0}$ and $M(\mathbf{\lambda},\mathbf{X}(\mathbb{T}_f))$ have an $\alpha$-$\fre$  distribution, they are equally distributed if their scale parameters coincide, i.e.  $S_{\mathbf{X}(\mathbb{T}_f)}(\mathbf{\lambda})=S_{X_{t_0}}(1)=1.$ Recall that  $S^\alpha_{\mathbf{X}(\mathbb{T}_f)}(\mathbf{\lambda})=\|\mathbf{\lambda}^\alpha\|_{D(\mathbf{Z})}$, where $\|\cdot\|_{D(\mathbf{Z})}$ is a $D$-norm  with a generator $\mathbf{Z}=(Z_{t_j},t_j\in \mathbb{T}_f).$ Thus, $$\Lambda_L=\{\mathbf{\lambda}\in \bR_+^n:\|\mathbf{\lambda}^\alpha\|_{D(\mathbf{Z})}=1\}$$ is the $1/\alpha$-power transformation (applied coordinatewise) of the nonnegative quadrant of the unit $D(\mathbf{Z})-$ball $\mathbb{S}^+_{D(\mathbf{Z})}=\{\mathbf{\lambda}\in \bR_+^n:\|\mathbf{\lambda}\|_{D(\mathbf{Z})}=1\}.$

By Corollary \ref{cor:exMetr}, we can rewrite the minimization functional in \eqref{constraint1} in terms of characteristics of the $(n+1)$-variate random vector $\mathbf{X}_{c}:=\mathbf{X}(\mathbb{T}_c)=(X_{t_0},X_{t_1},\ldots,X_{t_n}),$ where $\mathbb{T}_c=(t_0,\mathbb{T}_f)\in \bZ_h^{n+1}.$ Then for $\mathbf{\lambda}  \in \Lambda_L,$  relation \eqref{Em:MC3:eq} yields  
\begin{equation}\notag
	\mathcal{E}_{H_\alpha}\left(X_{t_0},\hat{X}_\mathbf{\lambda}\right)=\rho_{\mathbf{X}_{c}}((1,\mathbf{0}),(0,\mathbf{\lambda}))=1-\frac{2}{l_{\mathbf{X}_c}((1,\mathbf{\lambda}^\alpha))+1},
\end{equation}
 where 
 $\mathbf{0}:=(0,\ldots,0), \mathbf{\lambda}=(\lambda_1,\ldots,\lambda_n) \in\bR_+^{n}$. 

Applying the latter relation, we reformulate the law-preserving prediction as follows:
\begin{proposition}\label{l_corollary}
 The optimal weight vector ${\hat{\mathbf{\lambda}}} =(\hat{\lambda}_1,\ldots,\hat{\lambda}_n) $ from Definition \ref{Xp_definition} has the form 
	\begin{equation}\label{min_problem_l}
	\hat{\mathbf{\lambda}}=\argmin \limits_{\mathbf{\lambda}\in\bR_+^{n}}\{	l_{\mathbf{X}_c}((1,\mathbf{\lambda}^\alpha)):\quad \,l_{\mathbf{X}_c}((0,\mathbf{\lambda}^\alpha))=1 \}.
	\end{equation}
\end{proposition}

Now we are going to restate Proposition \ref{l_corollary} in terms of $D$-norms taking into account their connection to tail dependence functions and the scale parameter $S_{\mathbf{X}_c}$. To this end, introduce the notation $\mathbf{w}=\mathbf{\lambda}^\alpha=(\lambda_1^\alpha,\ldots, \lambda_n^\alpha)$.

\begin{remark}[Geometrical interpretation]\label{remGeomInt}
Let $\mathbf{Z}$ be a generator of $(n+1)$-variate D-norm $\|\cdot\|_{D(\mathbf{Z})}$ associated with the tail dependence function $l_{\mathbf{X}_{c}}$ of the random vector $\mathbf{X}_{c}$.  The minimization problem \eqref{min_problem_l}  has the following geometric formulation: 
\begin{equation}\label{min_problem_dnorm}
		\| (1,\mathbf{w})\|_{D(\mathbf{Z})} \to \min\limits_{\mathbf{w} \in \mathbb{S}^+_{D(\mathbf{Z})}},   
	\end{equation}
    where  $\mathbb{S}^+_{D(\mathbf{Z})}= \{ \mathbf{w} \in \bR^n_+: \| (0,\mathbf{w}) \|_{D(\mathbf{Z})} = 1\}$.
	In other words, the solution  $\widehat{\mathbf{w}}$  to the problem \eqref{min_problem_dnorm} lies on a convex surface that is part of the $(n-1)$-dimensional unit sphere with respect to the $D(\mathbf{Z})$-norm $\|\cdot\|_{D(\mathbf{Z})}$. Similarly,  the prediction weights lie on the boundary of an ellipsoid in the Gaussian case \cite{das2022}. If $(X_{t_0},X_{t_1},\ldots,X_{t_n})$ is a subgaussian random vector with stability index $\alpha\in(0,2)$ and i.i.d. standard Gaussian components, then $\Lambda_L$ is also a unit sphere in $\bR^n$ with respect to the Euclidean norm $\|\cdot\|_2$ (see \cite{spodarev2022}).	

     The corresponding predictor $\hat{X}_\mathbf{\lambda}$ of   $X_{t_0}$ writes now $\hat{X}_\mathbf{\lambda}=M(\widehat{\mathbf{w}}^{1/\alpha},\mathbf{X}(\mathbb{T}_f))$.
	 If $\alpha>1$, we may take the generator  $\mathbf{Z}=\left( \Gamma(1-1/\alpha) \right)^{-1}\mathbf{X}_c$. Then the target functional in \eqref{min_problem_dnorm} can be replaced by the expectation $\bE(X_{t_0}\vee \eX_\mathbf{\lambda})$, which is finite in this case.
\end{remark}
\begin{remark}[Polar coordinates formulation]
The constraint $\mathbf{w} \in \mathbb{S}^+_{D(\mathbf{Z})}$ in \eqref{min_problem_dnorm} can be reduced to a unit simplex (using the norm $\|\cdot\|_1$) or, more generally, a  positive quadrant $\mathbb{S}^+_n$ of a unit sphere with respect to an arbitrary norm $\|\cdot\|$. To see this, rewrite the problem \eqref{min_problem_dnorm} in terms of the tail dependence function $l_{\mathbf{X}_c}$:
\begin{equation}\label{min_problem_dep_fct}
\hat{\mathbf{w}}:=\argmin \limits_{\mathbf{w} \in \bR^{n}_+  }
    \{ l_{\mathbf{X}_c}\left( (1,\mathbf{w})\right): \; l_{\mathbf{X}_c}\left( (0,\mathbf{w})\right)=1 \}.  
	\end{equation}
Change variables  $\mathbf{w}=r \mathbf{z}$ to polar coordinates with respect to a norm $\|\cdot\|$ in $\mathbb{R}^n_+$, where $r=\|\mathbf{w}\|$ and $\mathbf{z}=r^{-1}\mathbf{w}$.
Use the homogeneity of order one of $l_{\mathbf{X}_c}(\cdot)$   on $\bR^{n+1}_+$ to write

$l_{\mathbf{X}_c}\left( (0,\mathbf{w})\right)=l_{\mathbf{X}_c}\left( (0,r\mathbf{z})\right)
=r l_{\mathbf{X}_c}\left( (0,\mathbf{z})\right)=1,$ from which we conclude $r^{-1}=l_{\mathbf{X}_c}\left( (0,\mathbf{z})\right)$.
    Consequently, the optimal weights in \eqref{min_problem_dep_fct} get the form $$\hat{\mathbf{w}}=l^{-1}_{\mathbf{X}_c}\left( (0,\hat{\mathbf{z}})\right)\hat{\mathbf{z}}$$  with
    \begin{equation}
    \label{min:prob:z}
        \hat{\mathbf{z}}:=\argmin \limits_{\mathbf{z}\in \mathbb{S}^+_n }
    \{	l^{-1}_{\mathbf{X}_c}\left( (0,{\mathbf{z}})\right) l_{\mathbf{X}_c}\left( (l_{\mathbf{X}_c}\left( (0,{\mathbf{z}})\right),{\mathbf{z}})\right) \} ,
    \end{equation} 
    which yields an equivalent formulation of \eqref{min_problem_dnorm}.
\end{remark}
\begin{remark}[Nonconvexity]	In general, the optimization problems \eqref{min_problem_l}-\eqref{min_problem_dep_fct} are not convex, compare \cite{yau1974}.
	However, their convexity  is granted if $\mathbb{S}^+_{D(\mathbf{Z})}$ is a unit simplex, i.e., 
    $X_{t_1},\ldots,X_{t_n}$ are stochastically independent. In this case, it holds $l_{\mathbf{X}_c}\left( (0,\mathbf{w})\right)=w_1+\ldots+w_n$, $\mathbf{w}=(w_1,\ldots,w_n)\in \bR^n_+$.
\end{remark}

If the exact shape of $\mathbb{S}^+_{D(\mathbf{Z})}$ is not explicitly known, we have to find a work-around. First, notice that the condition $X_{t_0} \stackrel{d}{=} \hat{X}_\mathbf{\lambda}\sim H_\alpha$ rewrites
\begin{equation}\notag
	 \exp(-X_{t_0}^{-\alpha}) \stackrel{d}{=} \exp(-\eX_\lambda^{-\alpha})\sim {\cal U}[0,1].
\end{equation} 
Thus, we can replace the constraint in \eqref{constraint1} by adding a penalty term using the \textit{squared 2-Wasserstein distance} \eqref{eq:2Wasserst} given in Section \ref{subsec23}. Although the use of 1-Wasserstein distance is also possible in this context, it is numerically inconvenient due to the absolute value inside of \eqref{eq:1Wasserst} which adds an extra non-differentiability issue to the stochastic gradient descent methods from Section \ref{sec6}.

\begin{definition}\label{eX_definition}
	The \textit{max-stable predictor} of $X_{t_0}$ is given by
	 \begin{equation}\notag
		\eX_{\hat{\mathbf{\lambda}}} = M(\hat{\mathbf{\lambda}},\mathbf{X}(\mathbb{T}_f)) = \max\limits_{j=1,\ldots,n} \{\hat{\lambda}_j X_{t_j}\},
 \end{equation}
	where the vector of optimal weights $\hat{\mathbf{\lambda}}=(\hat{\lambda}_1,\ldots, \hat{\lambda}_n)$ is a solution to minimization problem 
\begin{equation}\begin{split}\label{constraint2}
	\Psi(\lambda) := & 2\bE\left(e^{-(X_{t_0} \vee \eX_\lambda)^{-\alpha}}\right) - \bE\left(e^{-\eX_\lambda^{-\alpha}}\right) - 1/2 \\ + & \gamma\left(\frac{1}{3} - \bE\left(e^{-\eX_\lambda^{-\alpha}} \vee Y \right) + \bE\left(e^{-2 \eX_\lambda^{-\alpha}}\right)\right) \longrightarrow \min\limits_{\mathbf{\lambda}\in\bR^n_+}, \end{split}
\end{equation}
	 $Y$ is an independent copy of $e^{-\eX_\mathbf{\lambda}^{-\alpha}}$, and $\gamma>0$ is a penalty weight.  If it is not possible to generate $Y$, one can use an alternative form  of the 2-Wasserstein metric in \eqref{eq:2Wasserst} and write
	  \begin{equation}
	\begin{split}
		\hat{\mathbf{\lambda}} &= \argmin\limits_{\mathbf{\lambda}\in\bR^n_+} \left\{2\bE\left(e^{-(X_{t_0} \vee \eX_\mathbf{\lambda})^{-\alpha}}\right) - \bE\left(e^{-\eX_\mathbf{\lambda}^{-\alpha}}\right) -1/2 \right. \\
        &+ \left. \gamma \left(\frac{1}{3}+ \int\limits_0^1 F(u) \left(F(u)- 2u\right) du \right) \right\}
	\end{split} \label{constraint3} \end{equation}
	with $F$ being the distribution function of $e^{-\eX_\mathbf{\lambda}^{-\alpha}}$.
\end{definition}
The constants $1/2$ and $\gamma/3$ above are completely irrelevant for the minimization and are kept only to preserve the excursion metric range and the non-negativity of the penalty function. 
 As already mentioned in \eqref{eqScaleM}, the predictor $\eX_{{\mathbf{\lambda}}}$ is $H_\alpha(\cdot,S_{\mathbf{X}(\mathbb{T}_f)}(\mathbf{\lambda}))$-distributed
	for any $\mathbf{\lambda} \in \bR^n_+$.

Using the substitution $w_j=\lambda_j^\alpha$, $j=1,\ldots, n$ as above, let us reformulate this  problem in terms of the function $l_{\mathbf{X}_c}$. The following lemma is a direct consequence of relations  \eqref{Em:MC3:eq} and \eqref{omMS3}:

\begin{lemma}\label{lem:WD}
The minimization problem \eqref{constraint3} is equivalent to 
\begin{equation}\label{min_problem_3}
\begin{split}
\Psi_1({\bf w})&:=\frac{1}{l_{\mathbf{X}_c}\left((0, {\bf w})\right)+1} -\frac{2 }{l_{\mathbf{X}_c}\left((1, {\bf w})\right)+1}\\
&+ \frac{\gamma}{3}\cdot  \frac{2\left(l_{\mathbf{X}_c}\left((0,{\bf w})\right) -1\right)^2}{\left( 2l_{\mathbf{X}_c}\left((0,{\bf w})\right)+1\right) \left( l_{\mathbf{X}_c}\left((0,{\bf w})\right)+2\right) }  \to \min\limits_{{\bf w}\in\bR^n_+},\end{split}
\end{equation}
	where ${\bf w}=(w_1,\ldots,w_n)$ and $l_{\mathbf{X}_c}$ is the tail dependence function of $(X_{t_0}, X_{t_1},\ldots, X_{t_n})$.
\end{lemma}
The above lemma will be used to examine the non--uniqueness of forecasts in Section \ref{sec5}. For the numerical prediction of $X_{t_0}$, we design methods based on formulations  \eqref{constraint2}-\eqref{constraint3} by approximating the involved expectations by finite sums.

Let the random field $\{X_t: t\in T\}$ be ergodic. Our forecast method makes predictions of $X_{t_0}$ using $N$ \textit{learning samples} $\{X_t:t- k_j \in \mathbb{T}_f\}$, where $k_j$ are some shifts such that $\mathbb{T}_f+k_j \subset \mathbb{T}_0 \setminus \mathbb{T}_f$,  $j=1,\ldots,N$. To calculate the forecast weights $\hat{\lambda}$ numerically, we replace the expectations in \eqref{constraint2} or \eqref{constraint3} with their empirical counterparts and find an approximation of $\hat{\lambda}$ as the solution of a minimization problem with respect to the function $\Phi$:
\begin{equation}\label{min_problem}
	\Phi(\lambda) := \frac{1}{N}\sum\limits_{j=1}^{N} Q_j(\lambda)  \to \min\limits_{\lambda\in\bR^n_+},
\end{equation}
where
\begin{equation}\notag
	\begin{split}
		Q_j(\lambda):&= 2e^{-(X_{t_0+k_j} \vee M(\lambda,\mathbb{T}_f+k_j))^{-\alpha}}-e^{-M(\lambda,\mathbb{T}_f+k_j)^{-\alpha}} -1/2\\
		&+\gamma\biggl(\frac{1}{3}-\left(e^{-M(\lambda,\mathbb{T}_f+k_j)^{-\alpha}}\vee Y_j\right)+e^{-2M(\lambda,\mathbb{T}_f+k_j)^{-\alpha}}\biggr)
	\end{split}
\end{equation}
with $Y_j$ being independent copies of $e^{-\eX_\lambda^{-\alpha}}$ in case of formulation \eqref{constraint2} and
\begin{equation}\notag
	\begin{split}
		Q_j(\lambda) &:= 2e^{-(X_{t_0+k_j} \vee M(\lambda,\mathbb{T}_f+k_j))^{-\alpha}}-e^{-M(\lambda,\mathbb{T}_f+k_j)^{-\alpha}}-1/2+\gamma/3 \\ &+\gamma e^{-2M(\lambda,\mathbb{T}_f+k_j)^{-\alpha}} \\
		&- \frac{\gamma}{N}\biggl(e^{-M(\lambda,\mathbb{T}_f+k_j)^{-\alpha}} + 2 \sum\limits_{m=1}^{j-1}e^{-M(\lambda,\mathbb{T}_f+k_m)^{-\alpha}} \vee e^{-M(\lambda,\mathbb{T}_f+k_j)^{-\alpha}} \biggr),
	\end{split}
\end{equation}
if we use formulation \eqref{constraint3}. 

\begin{remark}\label{bootstrapping_remark}
	When performing numerical minimization of \eqref{constraint2}, one way to obtain $N$ independent copies $\{Y_1,\ldots,Y_N\}$ of $e^{-\eX_\lambda^{-\alpha}}$ is bootstrapping. For this, we sample $h_1,\dots,h_N$ from $\{1,\dots,N\}$ with replacement and set $B_j=\mathbf{X}(\mathbb{T}_f+k_{h_j})$. Afterwards, we can calculate
	\begin{equation}\notag
		Y_j = \exp\left(-M( \lambda, B_j)^{-\alpha}\right), \hspace{0.2cm} j = 1,\ldots,N,
	\end{equation}	 
	for a given $\lambda$.
\end{remark}


\begin{remark}\label{remark:mse}
	The penalty weight $\gamma>0$ can be tuned numerically if multiple independently simulated samples $(X_{t_0}^{(j)},X^{(j)})=(X_{t_0}^{(j)},X_{t_1}^{(j)},\dots,X_{t_n}^{(j)})$, $j=1,\dots,K$ of $(X_{t_0}, X_{t_1},\ldots, X_{t_n})$  are available. Introduce the empirical distribution functions $\bar{F}_{\hat{U}_\lambda}$,  $\bar{F}_{\hat{U}_{\hat\lambda},\gamma}$ of $\hat{U}_{{\lambda}}=H_{\alpha}(\eX_{{\lambda}})$, $\hat{U}_{\hat{\lambda}}=H_{\alpha}(\eX_{\hat{\lambda}})$ by
    $$
    \bar{F}_{\hat{U}_\lambda}(x)=\frac1K \sum\limits_{j=1}^K {\One}\left\{H_{\alpha}(\eX_{{\lambda}}^{(j)})\le x\right\}, \quad x\in(0,1),
    $$
    \begin{equation} \label{eq:F_U}
     \bar{F}_{\hat{U}_{\hat\lambda},\gamma}(x)=\frac1K \sum\limits_{j=1}^K {\One}\left\{H_{\alpha}(\eX_{\hat{\lambda}}^{(j)})\le x\right\}, \quad x\in(0,1),   
    \end{equation}
    where $\eX_{{\lambda}}^{(j)}$, $\eX_{\hat{\lambda}}^{(j)}$ are the predicted values of $X_{t_0}$ based on the $j$th sample $X^{(j)}$, $j=1,\ldots, K$, calculated with fixed weights $\mathbf{\lambda}$ and with optimal weights $\hat{\mathbf{\lambda}}_j$.
    Since the parameter $\gamma$ controls the distance in probability law to $X_{t_0}$, it can be determined by minimizing the mean squared error between  $\bar{F}_{\hat{U}_{\hat{\lambda}},\gamma}$ and the cumulative distribution function of the continuous uniform distribution $\mathcal{U}[0,1]$ with respect to $\gamma$. To this end, define
$$\widehat{\text{MSE}}_\lambda:= \frac{1}{M}\sum\limits_{h=1}^M \left(  \bar{F}_{\hat{U}_{{\lambda}}} (x_h)-x_h\right)^2,$$
	where $x_h=h/M$, $h=1,\ldots, M$, $M\in\bN$ are  equidistant points on $(0,1]$. Setting $\widehat{\text{MSE}}(\gamma) :=\widehat{\text{MSE}}_{\hat{\lambda}}$, where $\hat{\mathbf{\lambda}}=\hat{\mathbf{\lambda}}_j$ with probabilities $1/K$, $j=1,\ldots,K$ (compare \eqref{eq:F_U}), we solve
the minimization problem  
	\begin{equation}\notag
		\widehat{\text{MSE}}(\gamma) = \frac{1}{M}\sum\limits_{h=1}^M \left(  \bar{F}_{\hat{U}_{\hat{\lambda},\gamma}}(x_h)-x_h\right)^2\longrightarrow \min_{\gamma>0}.
	\end{equation}

\end{remark}

\begin{remark}\label{rem:ecdf}
  For an  arbitrary, but fixed $\mathbf{\lambda}\in\bR^n_+$, the expectation and the variance of the empirical distribution function $\overline{F}_{\hat{U}_{\lambda}}$ of the random variable   $\hat{U}_{\lambda}=H_{\alpha}(\eX_{{\lambda}})$ at a fixed point $x\in(0,1)$ are given by
    $$
        \mathbb{E}\left[\overline{F}_{\hat{U}_{\lambda}}(x)\right]=x^{b_{\lambda}}, \quad
        \text{Var}\left[\overline{F}_{\hat{U}_{\lambda}}(x)\right]=\frac{1}{K}x^{b_{{\lambda}}} \left( 1-x^{b_{\lambda}}\right),
   $$
    where $b_{\lambda}=l_X(\lambda^{\alpha})$ and $l_X$ is the tail dependence function of $X^{(j)}=(X_{t_1}^{(j)},\dots,X_{t_n}^{(j)})$. Using this, we compute the expectation 
    \begin{align*}
        \mathbb{E}\left[\widehat{\text{MSE}_\lambda }\right]&= \frac{1}{MK}\sum_{h=1}^{M}  x_h^{b_{{\lambda}}} \left( 1-x_h^{b_{\lambda}}\right)+  \frac{1}{M}\sum_{h=1}^{M} \left( x_h^{b_{{\lambda}}}-x_h\right)^2 \\
        &\sim    \frac{1}{K} \int\limits_0^1 x^{b_{{\lambda}}} \left( 1-x^{b_{\lambda}}\right)\, dx  +   \int\limits_0^1 \left( x^{b_{{\lambda}}}-x\right)^2  \, dx \\
        &\longrightarrow \frac{2}{3} \frac{\left( b_{\lambda}-1\right)^2}{ \left( 2b_{\lambda}+1\right) \left( b_{\lambda}+2\right)}, \qquad M, K\to \infty,
    \end{align*}
     which is non--vanishing only if  $b_{\lambda}\neq 1$. Note that 
     $$\mathbb{E}\left[\widehat{\text{MSE}}(\gamma)\right]= \mathbb{E}\left( \mathbb{E}\left[\widehat{\text{MSE}}_\lambda   \mid  \lambda=  \hat{\lambda} \right] \right),$$ connecting the above calculation to Remark~\ref{remark:mse}. 
     If the prediction is close to law preserving (which happens for $\theta_X\approx 1$ or  large penalty weights $\gamma$) then $b_{\hat\lambda}\approx 1$ (compare formulation \eqref{min_problem_l}) and thus the asymptotic value of expected $\widehat{\text{MSE}}(\gamma)$ is close to zero.
\end{remark}


\section{Existence of the forecast}\label{sec4}

Now we check that the minimization problems  \eqref{min_problem_l}-\eqref{min_problem_3} have a solution, before trying to find their  numerical approximations in Sections \ref{sec6} and \ref{sec7}.
Although that follows from \cite[Theorem 5.1]{spodarev2022}, whose conditions (i)-(iii) are easily verifiable in the max-stable case, we give independent and shorter existence proofs in Theorems \ref{Thm:exist}  and \ref{main_existence_theorem} below. The uniqueness of the solution is in general not given. Here we refer to papers \cite{DavisResnick,DavisResnick93} as well as Section \ref{sec5} for examples.  

\begin{theorem}
     \label{Thm:exist}
    There exists a solution to minimization problems \eqref{min_problem_l}-\eqref{min:prob:z}.
\end{theorem}
\begin{proof} 
Using the equivalent formulation \eqref{min_problem_dep_fct} of the above minimization problems, we notice that the function $l_{\mathbf{X}_c}$ is convex  and thus continuous on $\mathbb{R}^{n+1}_+$.  By its continuity, the constraint set
$$\mathbb{S}^+_{D(\mathbf{Z})}=\{\mathbf{w} \in \bR^{n}_+  : \; l_{\mathbf{X}_c}\left( (0,\mathbf{w})\right)=1 \}$$
is closed in $\bR^{n}_+$. Due to the inequality \eqref{eq:Ineql}, 
$\mathbb{S}^+_{D(\mathbf{Z})}$ is a subset of the unit cube $\{\mathbf{w} \in \bR^{n}_+  : \; \|\mathbf{w}\|_\infty\le 1 \}$ and thus bounded. Hence, the continuous function $l_{\mathbf{X}_c}\left( (1,\mathbf{w})\right)$ attains its minimum on the compact $\mathbb{S}^+_{D(\mathbf{Z})}$.
\end{proof}
Let us concentrate on the minimization problems  \eqref{constraint2}-\eqref{min_problem_3}.
 Since they are all equivalent, we are free to choose any of them to work with. For instance, it holds
$\Psi_1({\bf w}) = \Psi({\bf w}^{1/\alpha}) - 1/2$,
where $\Psi$ and $\Psi_1$ are the functionals from \eqref{constraint2} and \eqref{min_problem_3}, respectively. Now we prove the main result of this section. 



\begin{theorem}\label{main_existence_theorem}
	The minimization problems \eqref{constraint2}-\eqref{min_problem_3}
	have a solution.
\end{theorem}
	\begin{proof}
The tail dependence function $l_{\mathbf{X}_c}$ is  continuous. Since $\Psi_1$ is a rational function of $l_{\mathbf{X}_c}$, it is also  continuous on  $\bR^n_+$.
By inequality \eqref{eq:Ineql} and easy calculations, it holds 
\begin{eqnarray*}
			\lim\limits_{{\bf w} \searrow 0} l_{\mathbf{X}_c}\left((0, {\bf w})\right) = 0, \quad 
            \lim\limits_{{\bf w} \searrow 0} l_{\mathbf{X}_c}\left((1, {\bf w})\right) = 1, \\
             \lim\limits_{\|{\bf w}\|_\infty \to \infty}l_{\mathbf{X}_c}\left((0, {\bf w})\right) = \lim\limits_{\|{\bf w}\|_\infty \to \infty}l_{\mathbf{X}_c}\left((1, {\bf w})\right) = \infty, \\
             l_{\mathbf{X}_c}\left((0, {\bf e}_1)\right)=1, \quad  l_{\mathbf{X}_c}\left((1, {\bf e}_1)\right)=\theta_{\mathbf{X}_c}^{\{0,1\}}\in [1,2],
\end{eqnarray*} 
and thus
\begin{eqnarray*}
			\lim\limits_{{\bf w} \searrow 0} \Psi_1({\bf w}) = \lim\limits_{\|{\bf w}\|_\infty \to \infty} \Psi_1({\bf w})  = \gamma/3, \\
             \Psi_1({\bf e}_1)  =\frac{\theta_{\mathbf{X}_c}^{\{0,1\}}-3}{2(\theta_{\mathbf{X}_c}^{\{0,1\}}+1) }\le - 1/4<\gamma/3.
\end{eqnarray*} 
Hence, the minimum of $\Psi_1$ is attained within a bounded set, i.e., there exists an $M\in(0,\infty)$ with
			\begin{equation}\notag
				\min\limits_{{\bf w}\in\bR^n_+} \Psi_1({\bf w}) = \min\limits_{{\bf w}\in[0,M]^n} \Psi_1({\bf w}).
			\end{equation}
			To summarize, the continuous function $\Psi_1$ attains its minimum on the compact set $[0,M]^n$, which justifies the existence of a solution to  minimization problems \eqref{constraint2}-\eqref{min_problem_3}.
	\end{proof}

\section{Non-uniqueness of the forecast}\label{sec5}

It is a common knowledge that the metric projection onto a subspace $E$ of a functional space $L^1[0,1]$ with the metric  induced by the norm  $\| \cdot\|_{L^1[0,1]}$ is not unique. Moreover, there can exist infinitely many functions from $E$ realizing this projection.

Recall that $l_{\mathbf{X}_c}$ is the tail dependence function of $(X_{t_0}, X_{t_1},\ldots, X_{t_n})$. Since it is an $L^1[0,1]$-norm of a max-linear combination \eqref{eq:l_w_int} of probability densities $p_j$, it is natural to expect the same situation with the non-uniqueness of our forecast. Moreover, since the (non-constrained) metric projections with respect to excursion metric and Davis-Resnick distance are equivalent (cf. Section \ref{Subsect:DR}), the non-uniqueness example of a MARMA(1,1)-process from  \cite[Section 4]{DavisResnick} applies also to our case.

We now make a more general statement about the non-uniqueness of the solution of problems \eqref{min_problem_l}-\eqref{min_problem_3}.
\begin{lemma}\label{lem:NonU0}
	If  $l_{\mathbf{X}_c} \left((w_0, {\bf w})\right)$ is symmetric with respect to $ {\bf w}\in \bR^n_+$ for $w_0=0$ and $w_0=1$,  	
	  then  the solutions to prediction problems \eqref{min_problem_l}-\eqref{min_problem_3}
	are not unique.
	\end{lemma}
	\begin{proof}
	It is clear that any permutation of the weights $\hat\lambda_j$ of a solution $\hat\lambda$ to optimization problems \eqref{min_problem_l}-\eqref{min_problem_3} is also their solution, leading to non-unique predictors $\eX_\lambda$.
	\end{proof}

\begin{theorem}\label{thm:NonU}
	Let $n\ge 2$. If there exists a non--decreasing continuous piecewise continuously differentiable function $\zeta:\bR_+\to\bR_+$ with $\lim_{y\to + \infty}\zeta^\prime(y)<\infty$ such that $l_{\mathbf{X}_c} \left((1, {\bf w})\right)=\zeta\left(  l_{\mathbf{X}_c} \left((0, {\bf w})\right) \right)$ for any $ {\bf w}\in \bR^n_+$,  	
	  then the prediction problems \eqref{min_problem_l}-\eqref{min_problem_3}
	have infinitely many solutions for any penalty weight $\gamma>0$.
	\end{theorem}
	\begin{proof}
	Using  the formulation \eqref{min_problem_dep_fct}, the target functional  $l_{\mathbf{X}_c} \left((1, {\bf w})\right)$ is constantly equal to $\zeta(1)$ on the set of ${\bf w}\in \bR^n_+$ such that  $l_{\mathbf{X}_c} \left((0, {\bf w})\right) = 1$, hence the problem \eqref{min_problem_l} has infinitely many solutions. As for the problem  \eqref{min_problem_3}, we may rewrite its target functional in Lemma \ref{lem:WD} as
	
    \begin{equation}\label{min_problem_4}    
	\Psi_{1}(y)= \frac{1}{y+1}-\frac{2}{\zeta(y)+1}  + \frac{\gamma}{3}\cdot \frac{2\left(y -1\right)^2}{\left( 2y+1\right) \left( y+2\right) },
	\end{equation}
	where $y=l_{\mathbf{X}_c} \left((0, {\bf w})\right)$. By homogeneity and continuity of $l_{\mathbf{X}_c}$, it holds $\zeta(0)=1$. Let us show that $\zeta(y)\sim y$ as $y\to + \infty$. 
    Since $l$ is convex and homogeneous of order $1$, it is subadditive.
Therefore, we write
\[
l_{\mathbf{X}_c}((1,\mathbf{w})) 
\leq l_{\mathbf{X}_c}((1,\mathbf{0})) + l_{\mathbf{X}_c}((0,\mathbf{w})) 
= 1 + y.
\]
Also, it holds
$
l_{\mathbf{X}_c}((1,\mathbf{w})) 
\geq l_{\mathbf{X}_c}((0,\mathbf{w})) 
= y
$
by coordinate monotonicity of $l$, which easily follows from representation \eqref{eq:l_w_int}. 
Hence, we get
\begin{equation}\label{eq:zeta}
    y \leq \zeta(y) \leq y + 1,
\end{equation}
which implies
$
\frac{\zeta(y)}{y} \to 1
$ for large $y$.
This yields $\lim_{y\to + \infty}\Psi_{1}(y)=\Psi_{1}(0)=\frac{\gamma}{3}$. 
     The derivative of the above functional  $\Psi_{1}$, where it exists, with respect to $y$ is given by	
	\begin{equation}\label{min_problem_4_derivative}
	{\Psi}^\prime_1(y)=	\frac{2\zeta'(y)}{(\zeta(y)+1)^2} - \frac{1}{(y+1)^2} +  \frac{6 \gamma \left(y^2 -1\right)}{\left(2 y+1\right)^2 \left( y+2\right)^2}.
	\end{equation}
    For exceptional points $y>0$, where $\zeta'$ does not exist, we may consider its left-hand as well as right-hand side derivatives $\zeta'_-(y)$, $\zeta'_+(y)$.
	Since $\zeta$ is non--decreasing, the first summand in the expression  \eqref{min_problem_4_derivative} is non--negative for all $y\ge 0$, while the third summand is  negative for $y\in[0,1)$ and  positive for $y>1$. Moreover, 
    it holds ${\Psi}^\prime_1(0)=\zeta'(0)/2-1-3/2 \gamma<0$ for all $\gamma>0$, since $0\le \zeta'(0)\le 1$ by upper bound \eqref{eq:zeta}.
    In addition, it can be shown ex adverso from the inequality \eqref{eq:zeta} that   $\lim_{y\to + \infty}\zeta^\prime(y)\le 1$. Hence, we get
    ${\Psi}^\prime_1(y)>0$ for every $\gamma>0$ and for  $y$  large enough as well as
    $\lim_{y\to + \infty} {\Psi}^\prime_1(y)= 0$.
Then,  for any $\gamma>0$ we have that either ${\Psi}^\prime_1$ has a zero $y_0\in(0,\infty)$, or $\lim_{y\to y_0-0} {\Psi}^\prime_1(y)<0$, $\lim_{y\to y_0+0} {\Psi}^\prime_1(y)>0$, compare Figure \ref{zeta_plot}.  It follows that \eqref{min_problem_4} has a  point of minimum $y_0$. Therefore, infinitely many solutions of \eqref{min_problem_3} belong to the set 	$\{  {\bf w}\in\bR_+^n: \; l_{\mathbf{X}_c}\left((0, {\bf w})\right) = y_0\}$.
	\end{proof}



    

\begin{example}\label{ex:NonU}
	Assume that the max--stable random field $\rf$ has  Fr\'echet$(\alpha)$ distributed marginals, $\alpha>0$. If $\{X_{t_1},\ldots ,X_{t_n} \}$ are exchangeable then, by Lemma \ref{lem:NonU0}, the prediction problems \eqref{min_problem_l}-\eqref{min_problem_3} have multiple solutions. To give a nontrivial example of this situation, consider an extremal Gaussian random field $\{G_t,\; t\in\bR^2\}$ from \eqref{extremalgaussian_proc_eq} in two dimensions, where the generating stationary Gaussian random field $\{ Y_t, \; t\in\bR^2\}$ is additionally  isotropic. Take the locations $t_1,\ldots, t_n \in\bR^2$ to be the vertices of a regular $n$-gon with its center of gravity at the origin  $t_0={\bf 0}$. Due to isotropy of $\{G_t,\; t\in\bR^2\}$, the random variables $\{G_{t_1},\ldots, G_{t_n} \}$ are exchangeable, and then the forecast of $G_{t_0}$ on the basis of the sample $\{G_{t_1},\ldots, G_{t_n} \}$ is not unique.

    Furthermore, if $\{X_{t_0},\ldots X_{t_n} \}$ are   i.i.d. or $X_{t_j}=X_0$ a.s. for any $i=0,\ldots,n$, then,  by Theorem  \ref{thm:NonU}, the prediction problems \eqref{min_problem_l}-\eqref{min_problem_3} have infinitely many solutions. Indeed,   it holds $l_{\mathbf{X}_c}((w_0,{\bf w})) = \| (w_0,{\bf w}) \|_1  $ for independent $X_{t_0},\ldots,X_{t_n}$, and  the assumptions of the above theorem are met with $\zeta(y)=y+1$. For instance, the set of solutions to the problem \eqref{min_problem_dep_fct} coincides with the whole simplex $\mathbf{S}_n$. 
		For the case of complete dependence $X_{t_j}=X_0$,  $i=0,\ldots,n$,  the tail dependence function reads
		$l_{\mathbf{X}_c} ((w_0,{\bf w})) = \| (w_0,{\bf w}) \|_{\infty}$ satisfying the conditions of Theorem \ref{thm:NonU} with $\zeta(y) = y \vee 1,$ see Figure \ref{zeta_plot}. Here, the same forecast    $\eX_\mathbf{\lambda }=  X_0 \cdot \left( \| \hat{\mathbf{\lambda}}\|_\infty \vee 1 \right)$ is realized at an infinite number of  $\hat{\mathbf{\lambda}}\in\bR^n_+$ satisfying $\| \hat{\mathbf{\lambda}}\|_\infty=1$.
			\end{example} 

\begin{figure}
	\begin{center}
	\includegraphics[scale=0.8]{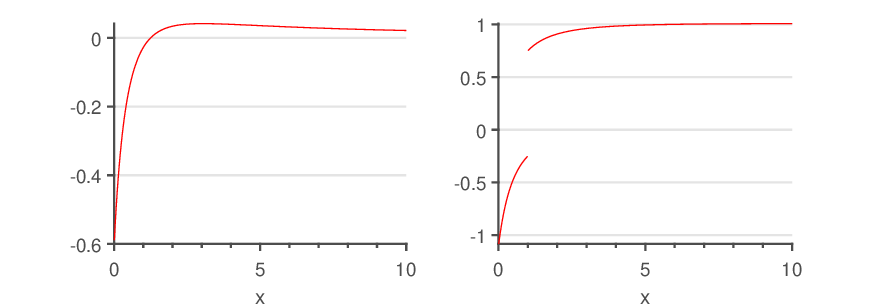}
	\end{center}
	\caption[Plots of the Derivatives of Target Functional]{The derivative ${\Psi}^\prime_1(y)$ of the function \eqref{min_problem_4} with $\gamma=1$ for  $\zeta(y) = y+1$ (left) and $\zeta(y) = y\vee 1$ (right), compare the proof of Theorem \ref{thm:NonU} and Example \ref{ex:NonU}.}\label{zeta_plot}
\end{figure}

\section{Minimization of target functionals via stochastic gradient descent}\label{sec6}

To approximate a solution of \eqref{min_problem}, we may use, alongside with a number of standard optimization methods, a stochastic  gradient descent (SGD). The advantage compared to the classical (batch) gradient descent (see \cite{saad1999}) lies in much faster computation. Both require the existence of a gradient but as one can quickly see, the function $\Phi$ is only piece-wise differentiable, since the summands $Q_j$ contain maximum functionals. Let us define the set of non-differentiability points of $Q_j$ by
\begin{subequations}\label{eq:ups}
\begin{align}
\Upsilon_j
&:=\left\{\mathbf{\lambda}\in\bR^n_+:X_{t_0+k_j}=M(\mathbf{\lambda},\mathbb{T}_f+k_j)\right\} \label{eq:ups1}\\
&\cup\left\{\mathbf{\lambda}\in\bR^n_+:\mathbf{\lambda}_jX_{t_j}=\mathbf{\lambda}_lX_{t_l}=M(\mathbf{\lambda},\mathbb{T}_f+k_m)\right. \notag\\
&\qquad\left.\text{ for some }t_j,t_l\in\mathbb{T}_f+k_m,t_j\neq t_l,m\le j\right\} \label{eq:ups2}\\
&\cup\left\{\mathbf{\lambda}\in\bR^n_+:M(\mathbf{\lambda},\mathbb{T}_f+k_j)=M(\mathbf{\lambda},B_j)\right\} \label{eq:ups3}\\
&\cup\left\{\mathbf{\lambda}\in\bR^n_+:\mathbf{\lambda}_j\tilde{X}_{t_j}=\mathbf{\lambda}_l\tilde{X}_{t_l}=M(\mathbf{\lambda},B_j)\right. \notag\\
&\qquad\left.\text{ for some }t_j,t_l\in\mathbb{T}_f+k_j,t_j\neq t_l\right\} \label{eq:ups4}\\
&\cup\left\{\mathbf{\lambda}\in\bR^n_+:M(\mathbf{\lambda},\mathbb{T}_f+k_m)=M(\mathbf{\lambda},\mathbb{T}_f+k_j)\right. \notag\\
&\qquad\left.\text{ for some }m<j\right\} \label{eq:ups5}
\end{align}
\end{subequations}
and put $\Upsilon:= \Upsilon_1\cup\ldots\cup \Upsilon_N$. Note that for the non-bootstrap version, the set of non-differentiability points of $Q_j$ is concentrated on the subsets \eqref{eq:ups1}, \eqref{eq:ups2}, and \eqref{eq:ups5}, whereas   it consists of the subsets \eqref{eq:ups1}, \eqref{eq:ups2}, \eqref{eq:ups3} and \eqref{eq:ups4} for the bootstrap version. Since the joint distribution of random variables $X_{t_0+k_j}$, $X_{t_j}$, $X_{t_l}$, $M(\mathbf{\lambda}, \mathbb{T}_f+k_j)$ is absolutely continuous, it holds $\bP(\Upsilon) = 0$. Hence, the gradient of $\Phi$ is almost surely given by $\nabla \Phi(\mathbf{\lambda}): =(1/N) \sum_{j=1}^N \nabla Q_j(\mathbf{\lambda})$, where the gradients of the summands $Q_j$ are  
\begin{equation}\notag
	\nabla Q_j(\mathbf{\lambda}) :=  \begin{pmatrix}
		\nabla q_j^{(1)}(\mathbf{\lambda}) - \nabla p_j^{(1)}(\mathbf{\lambda})\\
		\vdots\\
		\nabla q_j^{(n)}(\mathbf{\lambda}) - \nabla p_j^{(n)}(\mathbf{\lambda})\\
	\end{pmatrix}.
\end{equation}
In case of formulation \eqref{constraint2} we have
\begin{equation}\label{eq:q}
	\begin{split}
		\nabla q^{(i)}_j(\mathbf{\lambda}) := &\biggl[2\one\{X_{t_0+k_j} < \mathbf{\lambda}_i X_{t_i+k_j}\}-1\\
        &- \gamma \one\left\{M(\mathbf{\lambda}, B_j) < M(\mathbf{\lambda},\mathbb{T}_f+k_j)\right\}\\
		&+2\gamma \exp\left(-(\mathbf{\lambda}_i X_{t_i+k_j})^{-\alpha}\right)\biggr] h_{\alpha}\left(\mathbf{\lambda}_i X_{t_i+k_j}\right) X_{t_i+k_j} \\
		&\cdot\one\left\{\mathbf{\lambda}_i X_{t_i+k_j} = M(\mathbf{\lambda},\mathbb{T}_f+k_j)\right\},
	\end{split}
\end{equation}
and
\begin{equation}\label{eq:p}
	\begin{split}
		\nabla p^{(i)}_j(\mathbf{\lambda}) := &\gamma \one\left\{M(\mathbf{\lambda}, B_j) > M(\mathbf{\lambda}, \mathbb{T}_f+k_j)\right\}\\
		&\cdot h_{\alpha}(\mathbf{\lambda}_i \tilde{X}_{t_i+k_j}) \tilde{X}_{t_i+k_j}\one\left\{\mathbf{\lambda}_i \tilde{X}_{t_i+k_j} = M(\mathbf{\lambda}, B_j)\right\}\\
	\end{split}
\end{equation}
for $i = 1,\ldots,n$, $j=1,\ldots, N$, where
\begin{equation}\notag
	h_{\alpha}(x) := \alpha x^{-1-\alpha} \exp(-x^{-\alpha}), \hspace{0.2cm} x>0
\end{equation}	
is the probability density function of the Fr\'echet$(\alpha)$ distribution. For formulation \eqref{constraint3}, we obtain
\begin{equation}\notag
	\begin{split}
		\nabla q^{(i)}_j(\mathbf{\lambda}) :&= \biggl[2\one\{X_{t_0+k_j} < \mathbf{\lambda}_i X_{t_i+k_j}\} - 1 + 2\gamma \exp\left((\mathbf{\lambda}_i X_{t_i+k_j})^{-\alpha}\right) - \frac{\gamma}{N}\\
		&- \frac{2\gamma}{N} \sum\limits_{m=1}^{j-1} \one\left\{M(\mathbf{\lambda}, \mathbb{T}_f+k_m) < M(\mathbf{\lambda}, \mathbb{T}_f+k_j)\right\}\biggr]\\
		&\cdot h_{\alpha}\left(\mathbf{\lambda}_i X_{t_i+k_j}\right) X_{t_i+k_j}\one\{\mathbf{\lambda}_i X_{t_i+k_j} = M(\mathbf{\lambda},\mathbb{T}_f+k_j)\}
	\end{split}
\end{equation}
and
\begin{equation}\notag
    \begin{split}
        \nabla p_j^{(i)}(\mathbf{\lambda}):&= \frac{2\gamma}{N} \sum\limits_{m=1}^{j-1}  \one\left\{M(\mathbf{\lambda}, \mathbb{T}_f+k_m) > M(\mathbf{\lambda}, \mathbb{T}_f+k_j)\right\} \\
		 &\cdot h_{\alpha}\left(\mathbf{\lambda}_i X_{t_i+k_m}\right) X_{t_i+k_m}\one\{\mathbf{\lambda}_i X_{t_i+k_m} = M(\mathbf{\lambda},\mathbb{T}_f+k_m)\}
    \end{split}
\end{equation}
for $i = 1,\ldots,n$, $j =1,\ldots,N$.


Apart from the gradient, we also need an initial value $\mathbf{\lambda}^{(0)}\in\bR^n_+$, a step size $\eta>0$ and an update function $f_k:\bR^k\to\bR$ for every iteration. In the classical SGD, one sets $f_k(x_1,\ldots,x_k)=x_k$. Denote by $ {\cal U}\{1,\ldots,N\}$ the discrete uniform distribution on $\{1,\ldots,N\}$.
\begin{algorithm}
\caption{Stochastic Gradient Descent for minimization problem \eqref{min_problem}}\label{sgd_algorithm}
\begin{algorithmic}[1]
\Require $\mathbf{\lambda}^{(0)}\in\bR^n_+$, $\eta\in (0,\infty)$

\State $k\gets0$; $\Phi^*\gets\Phi(\mathbf{\lambda}^{(0)})$ 
\While{\text{Convergence Criterion is not fulfilled}}
	\State $j_k \sim {\cal U}\{1,\ldots,N\}$
	\State $\mathbf{\lambda}^{(k+1)} \gets \mathbf{\lambda}^{(k)} - \eta f_k\left(\nabla Q_{j_l}(\mathbf{\lambda}^{(l)}), l=1,\ldots,k\right)$
    \State \textbf{if} $\Phi(\mathbf{\lambda}^{(k+1)})<\Phi^*$ \textbf{then} $\Phi^*\gets \Phi(\mathbf{\lambda}^{(k+1)})$ 
    \State $k \gets k+1$
\EndWhile\\
\Return $\mathbf{\lambda}^*$
\end{algorithmic}
\end{algorithm}

For our numerical experiments, we choose $\mathbf{\lambda}^{(0)}=(1,\dots,1)\in\bR^n_+$ as a starting value because it  gives   the same significance to each observation of the forecast sample. Practically, classical SGD with step size $\eta=0.1$ works well   and tends to perform better than \textsf{R}'s standard solvers like \texttt{optim} and \texttt{nlminb} in the sense that a lower value $\Phi^*$ of the objective function is achieved. The drawback is that we observe a strong oscillation around $\Phi^*$. The choice of  optimal value of $\eta$ still remains open. Although, for small sample sizes $n$, standard solvers work reasonably well and are even faster than the stochastic gradient descent, they become inferior to SGD once $n$ is large. As a Convergence Criterion, we use early stopping with patience parameter $p$, i.e., the algorithm stops if $\Phi^*$ does not improve for $p$ consecutive iterations.    The convergence of SGD methods is already well studied, see e.g. papers \cite{saad1999,garrigos2024handbookconvergencetheoremsstochastic, reddi2019convergenceadam,JMLR:v23:20-1438}. Hence, we skip the convergence analysis here. 

From Section \ref{sec4} we know that a solution to our forecast problems always exists, but (as shown in Section \ref{sec5}) it does not have to be unique. In order to give motivation that the solution might be unique in some special cases, we performed experiments where we extrapolated one step ahead using a forecast sample with two observations, i. e. $n=2$. We evaluated the resulting objective function on a regular grid with $40000$ points and created surface plots that can be seen on   Figure \ref{sgd_plot} (left).  Figure \ref{sgd_plot} (right) shows the corresponding iteration steps of the Adam algorithm \cite{kingma2014adam} where $f_k$ depends on the whole gradient pre-history.   In case of larger forecast sample sizes $n$, it performs faster than the classical SGD.


Since the target functional $\Phi$ has to be minimized on $\bR^n_+$, Algorithm \ref{sgd_algorithm} needs an additional back--projection step lifting each $\mathbf{\lambda}^{(k+1)}$ to $\bR^n_+$. Alternatively, one could use the reparametrization $\mathbf{\lambda}=e^{\mathbf{\tau}}$ and find $\mathbf{\tau}^*=\argmin_{\mathbf{\tau}\in\bR^n}\Phi(e^{\mathbf{\tau}})$. The gradient with respect to $\mathbf{\tau}$ is given by $\nabla_{\mathbf{\tau}}\Phi(e^{\mathbf{\tau}})=\nabla_{\mathbf{\lambda}}\Phi(\mathbf{\lambda})\odot\mathbf{\lambda}$, where $\odot$ represents component-wise multiplication. We use the latter option because it leads to lower values of the target functional.

\begin{figure}[ht]
  \centering

  \begin{minipage}{0.49\linewidth}
    \centering
    \includegraphics[width=\linewidth]{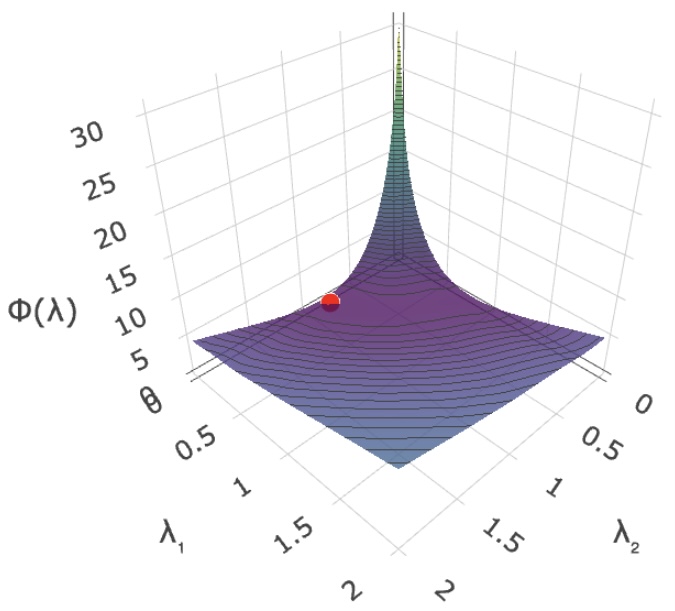}
  \end{minipage}\hfill
  \begin{minipage}{0.49\linewidth}
    \centering
    \includegraphics[width=\linewidth]{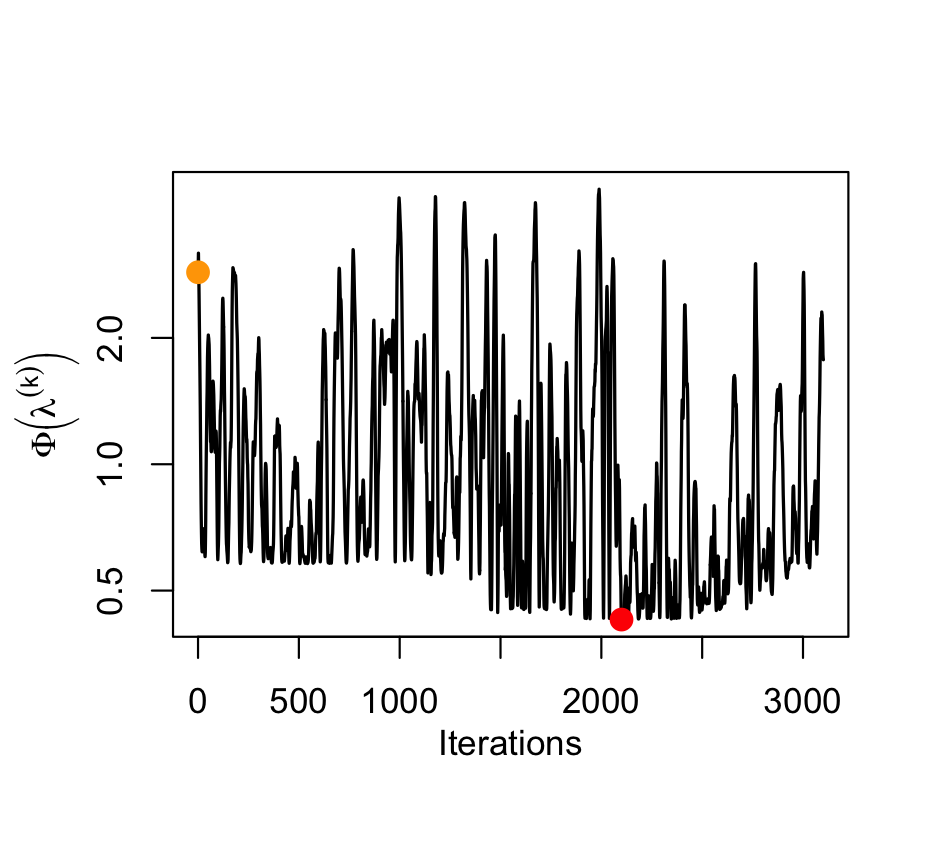}
  \end{minipage}

  \caption{The left plot shows the function $\Phi$ associated to the Brown-Resnick process $B$ from Example \ref{ex:BRP} with volatility $\sigma_B=1.68$, a forecast sample size of $n=2$, $t_1=201$, $t_2=202$, $t_0=203$, $\gamma=100$, $N=100$, and all even shifts $k_j$ from 2 to 198. The minimum value of the grid is located at $(\lambda_1,\lambda_2)=(0.08, 0.83)$ with $\Phi((\lambda_1,\lambda_2))\approx0.4265$. The right plot shows the iterations of the Adam algorithm, which achieved a minimum of $\Phi^*\approx0.4265$ at $(\lambda_1,\lambda_2)\approx(0.078, 0.823)$. The standard solver \texttt{nlminb} from \textsf{R} achieved a minimum of $\Phi^*\approx0.4275$ without being provided with the gradient of $\Phi$.}
  \label{sgd_plot}
\end{figure}

\section{Numerical examples}\label{sec7}  

For max-stable random fields mentioned in Section \ref{subsec24}, we predict their values for $d=1$ and $d=2$ from a simulated sample and also calculate the excursion metric to rate the goodness of this prediction. After that, we apply our forecasting algorithm to predict the extreme values of yearly precipitation in Munich based on observations from 1879 to 2025.

\subsection{Simulation of max-stable data}    


Let $X=\{X_t:t\in\bR\}$ be any of the three types of max-stable processes, which were introduced in Section \ref{subsec24} (Brown-Resnick, Smith, extremal Gaussian). For the Brown-Resnick case, we choose $C(s,t)=\sigma_B^2(s\wedge t)$ as covariance function of the underlying Gaussian process $\{ Y_t, t\in \bR \}$ in  \eqref{brownresnick_proc_eq}. In the extremal Gaussian case, we set the Ornstein-Uhlenbeck covariance function $C(s,t)=\exp{(-\lvert s-t\rvert/\sigma_G)}$ for the underlying Gaussian process  $\{ Y_t, t\in \bR \}$ in \eqref{extremalgaussian_proc_eq}. In the Smith case, we will denote by $\sigma_S^2$ the variance parameter of the normal probability density in \eqref{smith_proc_eq}. 

In order to make the extrapolation results comparable between the three types of processes, we choose the parameters $\sigma_B$, $\sigma_S$ and $\sigma_G$ in such a way that the extremal coefficients $\theta_B$, $\theta_S$, $\theta_G$ of the processes coincide. For the process $X\in\{B,S,G\}$, we recall that  its extremal coefficient function is given by $\theta_X(h)=l_{(X_{t},X_{t+h})}((1,1))$. Using the formulae of the tail dependence functions in Section \ref{subsec24}, the extremal coefficient functions of the Brown-Resnick, Smith and extremal Gaussian processes are given by $\theta_B(h)=2F_0(\sigma_B\sqrt{h}/2)$, $\theta_S(h)=2F_0(\sigma_S^{-1} h/2)$,  $\theta_G(h)=1+\sqrt{1-\exp{(-h/\sigma_G)}}/\sqrt{2}$, respectively. The selected parameters with the corresponding extremal coefficient $\theta_X:=\theta_X(1)$ are given in Table \ref{tab:tab1}.

\begin{table}[htbp]
    \centering
    \begin{tabular}{@{}cccc@{}}
    \toprule
    $\theta_X$ & $\sigma_B$ & $\sigma_S$ & $\sigma_G$  \\
    \midrule
    1.3 & 0.771 & 1.298 & 5.039\\
    1.6 & 1.683 & 0.594 & 0.786\\
    1.7 & 2.073 & 0.482 & 0.256\\ \bottomrule
    \end{tabular}
    \caption{Selected covariance parameters $\sigma_X$ with corresponding extremal coefficients $\theta_X$ for $X\in\{B,S,G\}$.}
    \label{tab:tab1}
\end{table}


We will use \cite[Algorithm 1]{DombEngOest16} to simulate $X$ at locations $\{1,\dots,L\}$, where the length $L$ will depend on the forecast horizon (how far into the future we predict), the size of the forecast sample and the number of learning samples. For our experiments, we  fix the size of the forecast sample $\lvert\mathbb{T}_f\rvert$ to be  $L+1$. First, we perform pre-experiments using the setting $L=1$ to find a suitable choice for the penalty term $\gamma$. Afterwards, we will use those optimal $\gamma$ to test our extrapolation procedure in the setting $L=20$. 


\subsection{Choice of the penalty weight}\label{sec:gamma_choice}
We address the optimal choice of penalty weight $\gamma$ by finding the best tradeoff between the excursion metric and the mean square error (MSE, see Remark \ref{remark:mse}). Let $X$ be any of the three max-stable processes from Section \ref{subsec24}. We simulate $K$ samples $X^{(j)}=\{X^{(j)}_1,\dots,X^{(j)}_{203}\}$, $j=1,\dots,K$ where $X_{t_0}^{(j)}=X_{203}^{(j)}$, $X(\mathbb{T}_f)=\{X_{201}^{(j)},X_{202}^{(j)}\}$ and the remaining 200 observations are used to form 100 non-overlapping learning samples. We then compute $(\hat{\mathcal{E}}_{H_{1}}(X_{t_0},\hat{X}_{\hat{\mathbf{\lambda}}}), \;\gamma\in\{0,1,\dots,20\})$,  compare \eqref{eq:empExcM}, and $(\widehat{\text{MSE}}(\gamma), \;\gamma\in\{0,1,\dots,20\})$ from Remark \ref{remark:mse}. Finally, we pick 
\begin{align}\label{eq:gamma_choice}
    \gamma_{opt}=\arg\min\left(\max\left\{\overline{\hat{\mathcal{E}}_{H_{1}}(X_{t_0},\hat{X}_{\hat{\mathbf{\lambda}}})},\overline{\widehat{\text{MSE}}}(\gamma)\right\}\right),
\end{align}
where $\overline{\mathbf{x}}=(\mathbf{x}-\min \mathbf{x})/(\max \mathbf{x} - \min \mathbf{x})$ for any vector $\mathbf{x}$, and  minimum as well as maximum are understood coordinatewise. The results of these numerical experiments for $\theta_X=1.7$ are displayed  in Figure \ref{fig:gamma_vs_goal}.  There, the excursion metric grows with increasing $\gamma$ for all three types of processes $X$. Note that for the extremal Gaussian case, the excursion metric quickly reaches a value of 0.3, whereas the excursion metric curves for the Brown-Resnick case (left) and the Smith (center) case increase moderately to 0.27. The MSE tends to zero for large $\gamma$ indicating that the penalty term works as intended, compare Remark \ref{rem:ecdf}. The non--monotonicity of the MSE curve in the extremal Gaussian case arises because the ergodic theorem fails for the extremal Gaussian process, so the sample-based penalty term is an inconsistent proxy for the true Wasserstein distance, and moderate penalty weights can be counterproductive before large weights finally force law preservation.
The values $\gamma_{opt}$ 
for each of the three types of processes $X$ and extremal coefficient  values $\theta_X\in\{1.1,\dots,1.9\}$ are given in Table \ref{tab:gam_star}.    

\begin{figure}[htbp]
  \centering
  \begin{minipage}[b]{0.32\textwidth}
    \centering
    \includegraphics[width=\linewidth]{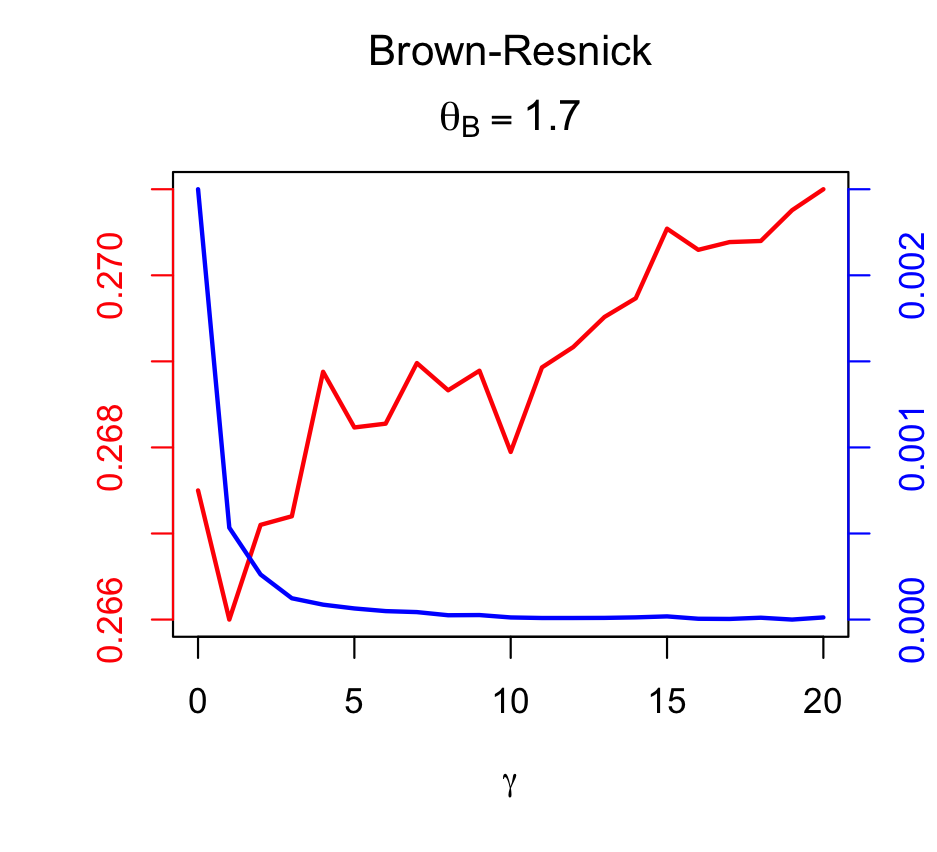}
  \end{minipage}\hfill
  \begin{minipage}[b]{0.32\textwidth}
    \centering
    \includegraphics[width=\linewidth]{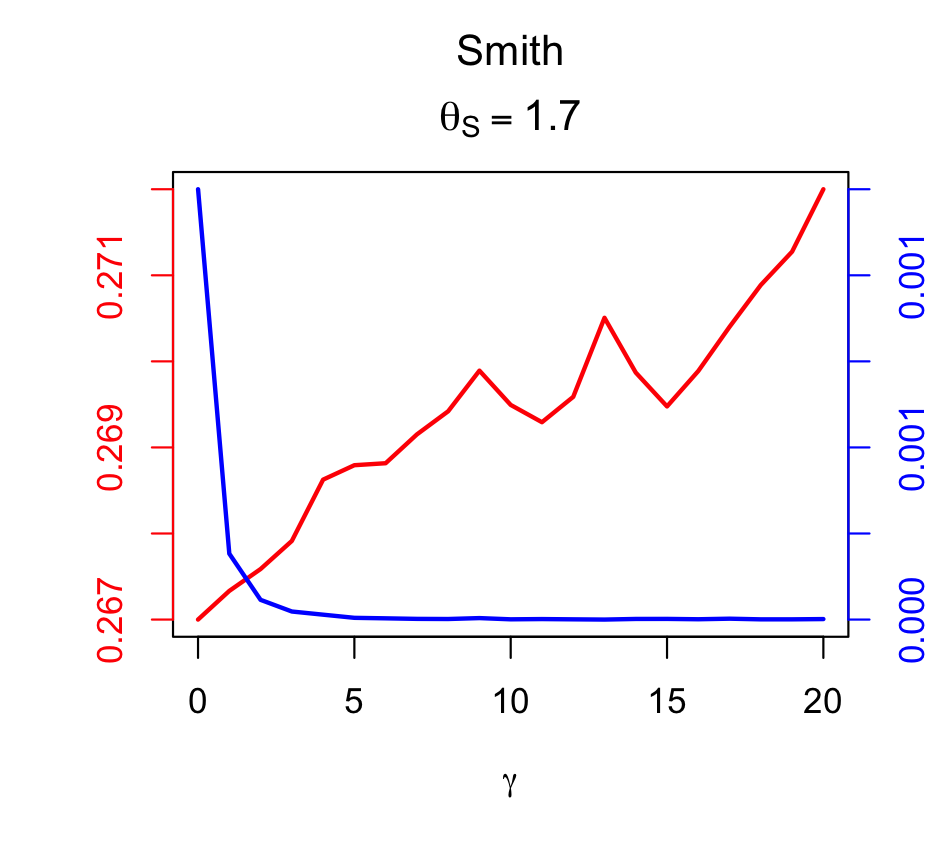}
  \end{minipage}\hfill
  \begin{minipage}[b]{0.32\textwidth}
    \centering
    \includegraphics[width=\linewidth]{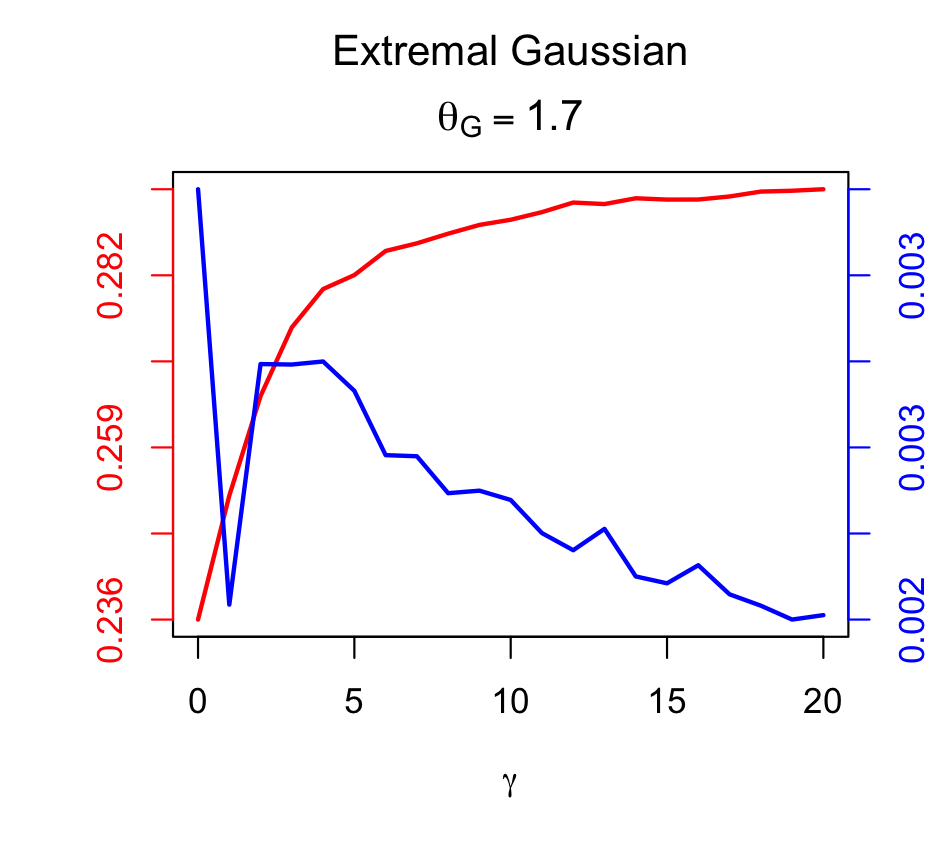}
  \end{minipage}

  \caption{The red line indicates the empirical excursion metric $\overline{\hat{\mathcal{E}}_{H_{1}}(X_{t_0},\hat{X}_{\hat{\mathbf{\lambda}}})}$ and the blue line represents $\overline{\widehat{\text{MSE}}}(\gamma)$ as functions of $\gamma$ for the three max--stable processes $X\in\{B,S,G\}$.  }
  \label{fig:gamma_vs_goal}
\end{figure}

\begin{table}[htbp]
\centering
\begin{tabular}{|c|c|c|c|c|c|c|c|c|c|}
\hline
Process type / \;  $\theta_X$ & 1.1 & 1.2 & 1.3 & 1.4 & 1.5 & 1.6 & 1.7 & 1.8 & 1.9 \\
\hline
Brown-Resnick  $B$    & 0 & 2 & 3 & 2 & 2 & 2 & 1 & 2 & 2 \\
\hline
Smith        $S$     & 0 & 3 & 0 & 3 & 2 & 5 & 2 & 2 & 1 \\
\hline
Extremal Gaussian $G$  & 0 & 0 & 0 & 0 & 0 & 1 & 1 & -  & -  \\
\hline
\end{tabular}
\caption{Choice of $\gamma_{opt}$ for each extremal coefficient $\theta_X\in\{1.1,\dots,1.9\}$ for three types of max-stable processes $X\in\{B,S,G\}$ from Section \ref{subsec24}. Note that the entries for 1.8 and 1.9 in line $G$ are missing because the extremal coefficient $\theta_G$ only lies in the interval $(1,v_l)$ with $v_l=1+1/\sqrt{2}\approx1.71$.}
\label{tab:gam_star}
\end{table}

After having determined optimal penalty weights $\gamma$, we  perform $L=20$-step prediction for each of the three types of processes $X$ using the non-bootstrap formulation \eqref{constraint3}. To do so, we use a forecast sample size of $n=21$ as well as $100$ non-overlapping learning samples of the same size. Therefore, we simulate processes of total time length $100\cdot21+21+20=2141$, where the last $20$ observations are used to compare them with the forecast to assess the quality of the prediction.

\subsection{Performance of max-stable prediction}\label{subsec71}

We performed 20 step predictions for the Brown-Resnick, Smith and extremal Gaussian processes, i.e., $X\in\{B,S,G\}$. Figures \ref{forecast_brownresnick}-\ref{forecast_extremalgaussian} (left) show the results for processes $X$ with extremal coefficient $\theta_X=1.3$ (high interdependence of observations), and the right sides show the results for processes with extremal coefficient $\theta_X=1.7$ (low interdependence of observations).

Using relation \eqref{Em:MC3:eq} from Corollary \ref{cor:exMetr}, rewrite the excursion metric under the law-preserving constraint (i.e., the Gini metric) in terms of the extremal coefficient function $\theta_X(\cdot)$:
\begin{equation}\label{eq:ExcM}
    \mathcal{E}_{H_1}(X_{t_n},X_{t_0})=1-\frac{2}{\theta_X(t_0-t_n)+1}.
\end{equation}
Figures \ref{forecast_brownresnick}-\ref{forecast_extremalgaussian} additionally contain the black curves $t_0\mapsto\mathcal{E}_{H_1}(X_{t_n},X_{t_0})$, which can be seen as a benchmark for the empirical  excursion metrics given in red.
Those are computed by replacing the expectations in \eqref{constraint2} with their empirical counterparts. To do so, a Monte-Carlo simulation using $K=1000$ stochastically independent processes $X^{(j)}$ and their predictions is performed. The empirical version of the excursion metric in \eqref{constraint2}   is then calculated as follows:
    \begin{align}\label{eq:empExcM}
        \hat{\mathcal{E}}_{H_{1}}(X_{t_0},\hat{X}_{\hat{\mathbf{\lambda}}})=\frac{2}{K}\sum_{j=1}^{K}H_{1}\left(X_{t_0}^{(j)}\vee \hat{X}_{\hat{\mathbf{\lambda}}}^{(j)}\right)-\frac{1}{K}\sum_{j=1}^{K}H_{1}\left(\hat{X}_{\hat{\mathbf{\lambda}}}^{(j)}\right)-\frac{1}{2} .
    \end{align}

Forecasting processes with highly dependent observations, we see similar phenomena: the first prediction steps have a low excursion metric, but then it quickly rises to around $0.3$ (recall that the Gini metric of $1/3$ indicates that the predictor is stochastically independent of the real process, and thus the prediction is not better than random). The Smith process reaches an excursion metric of $0.3$ already after the third step, whereas the Brown-Resnick process reaches this value only after the ninth step.

The extremal Gaussian process seemingly performs best in terms of the excursion metric: even after $20$ steps, it still has values of around $0.2$. However, this may be explained by the fact that its predictor $\widehat{G}_{\widehat{\mathbf{\lambda}}}$ is not exactly equal in distribution to $G_{t_0}$ due to non-ergodicity of $G$, and thus the excursion distance $\mathcal{E}_{H_1}(\widehat{G}_{\widehat{\mathbf{\lambda}}},G_{t_0})$ is not equal to the Gini metric. In such cases, its value $1/3$ is not characteristic of stochastic independence, and also lower values of $\mathcal{E}_{H_1}(\widehat{G}_{\widehat{\mathbf{\lambda}}},G_{t_0})$ are possible for nearly stochastically independent $\widehat{G}_{\widehat{\mathbf{\lambda}}}$ and $G_{t_0}$. Additionally, we get
\begin{align*}
    \mathcal{E}_{H_1}(G_{t_n},G_{t_0})&=1-\frac{2}{\theta_G(t_0-t_n)+1}\\&
    =1-\frac{2\sqrt{2}}{2\sqrt{2}+\sqrt{1-\exp{(-(t_0-t_n)/\sigma_G)}}}\\&
    {\longrightarrow}\frac{1}{2\sqrt{2}+1}\approx0.26, \quad t_0\to\infty,
\end{align*}
explaining the asymptotic value to which the theoretical excursion metric  converges for large lags $t$ (black curves in Figure~\ref{forecast_extremalgaussian}).


For processes $X$ with a higher extremal coefficient $\theta_X$, the prediction quality becomes worse, since the observations are almost independent of each other, and thus not much can be learned from previous observations. Thus,  the Brown-Resnick and Smith processes attain values of excursion metric around 0.29 already at the first step. As explained above, the extremal Gaussian process behaves again differently: its excursion metric does not rise to $1/3$ with increasing forecast distance, but it continues to oscillate between 0.21 and 0.24.

An extrapolation using formulation \eqref{constraint3} of 2D random fields $X$ of type $\{ B,S,G \}$ is shown in Figure \ref{forecast_random_field}. There, we observe the field $X$ at locations $\{1,\dots,n\}^2$ and then predict its values at spots
\begin{align*}
    (i,j)&\in\{1,\dots,n\}\times\{n+1,\dots,n+m\}\\
    &\cup\{n+1,\dots,n+m\}\times\{1,\dots,n\}\\
    &\cup\{n+1,\dots,n+m\}\times\{n+1,\dots,n+m\},
\end{align*}
where $m\ge1$ is the forecast horizon. For each target location $(i,j)$, the $m+1$ closest points (according to the Euclidean distance) in $\{1,\dots,n\}^2$ form the forecast sample. It may occur that multiple points in the grid $\{1,\dots,n\}^2$ have the same distance to the target location $z=(i,j)$ outside of the grid, but not all of them can be included in the forecast sample, since it has fixed size $m+1$. In such case, we compute the polar angles of the vectors $\overrightarrow{zz_0}$ for all points $z_0\in\{1,\dots,n\}^2$ that have the same distance to $z$ and use them as a secondary deterministic ordering criterion.

\begin{figure}[htbp]
\centering
\begin{minipage}[t]{0.48\textwidth}
  \centering
  \includegraphics[width=\linewidth]{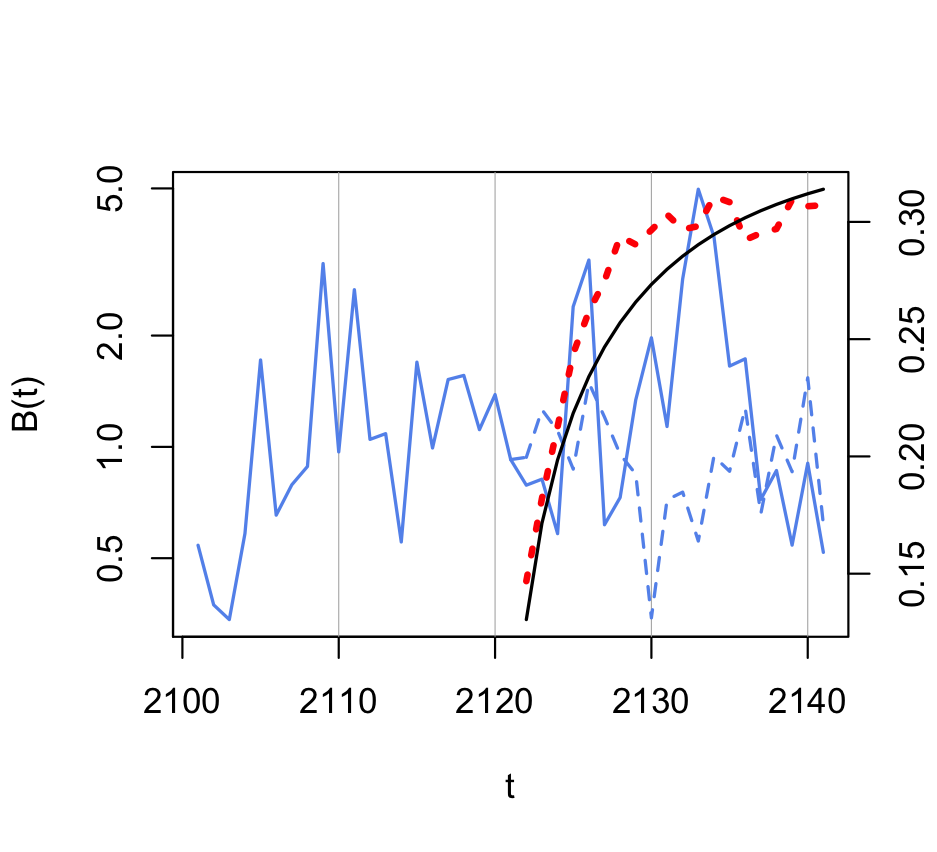}
\end{minipage}\hfill
\begin{minipage}[t]{0.48\textwidth}
  \centering
  \includegraphics[width=\linewidth]{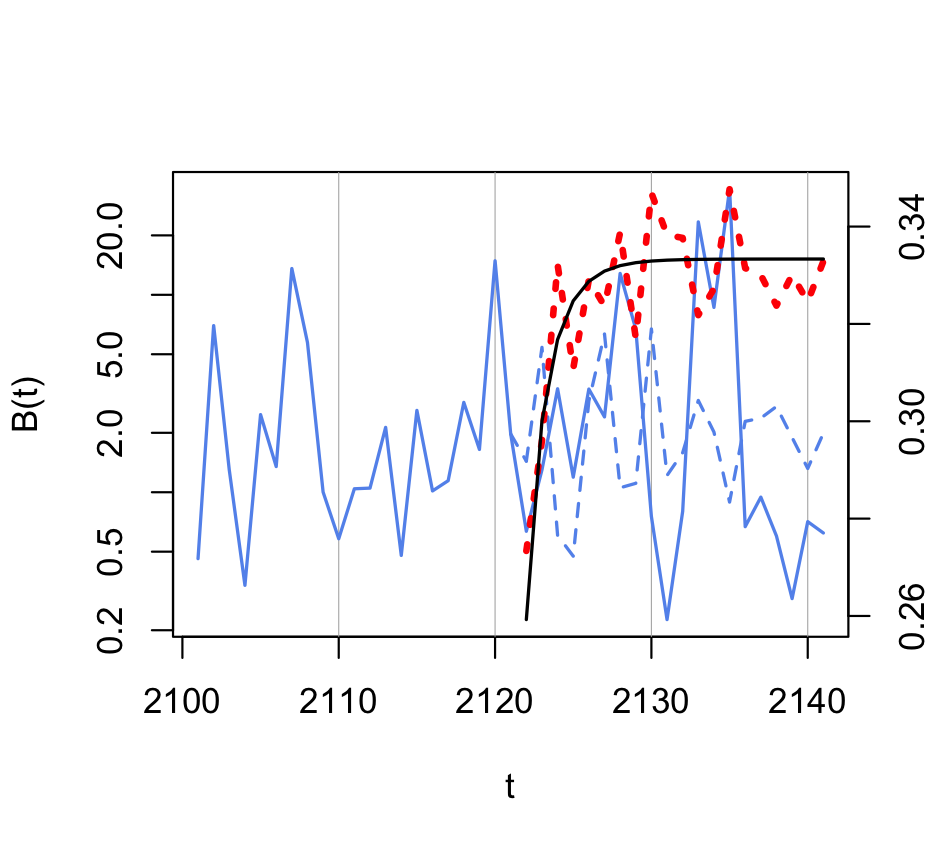}
\end{minipage}

\caption[Prediction of Brown-Resnick Process]{A 20 time steps' forecast (dashed blue line), together with its corresponding theoretical (black line) and empirical (red line) excursion metric, of a Brown-Resnick process $B$ (blue line). The underlying Gaussian process $Y$ of the Brown-Resnick process is a standard Brownian motion with variance parameter $\sigma_B = 0.771$ (left plot) and $\sigma_B=2.073$ (right plot). For the left plot, $\gamma_{opt}=3$ and for the right plot, $\gamma_{opt}=1$, cf. Table \ref{tab:gam_star}.}
\label{forecast_brownresnick}
\end{figure}

\begin{figure}[htbp]
\centering
\begin{minipage}[t]{0.48\textwidth}
  \centering
  \includegraphics[width=\linewidth]{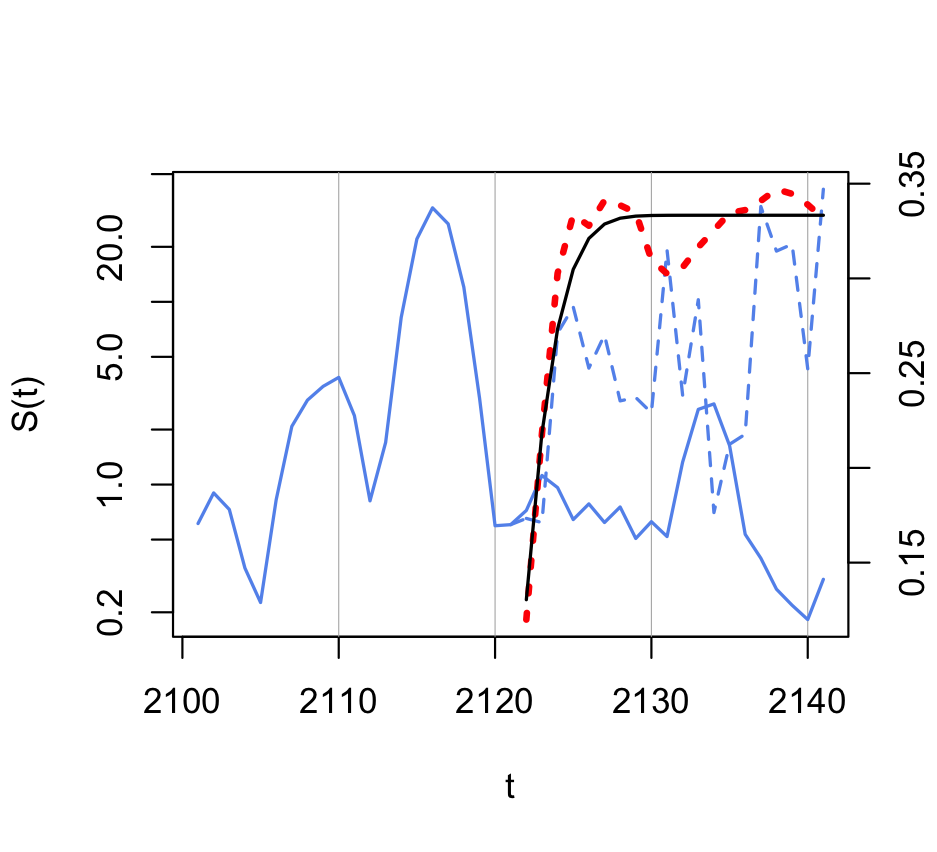}
\end{minipage}\hfill
\begin{minipage}[t]{0.48\textwidth}
  \centering
  \includegraphics[width=\linewidth]{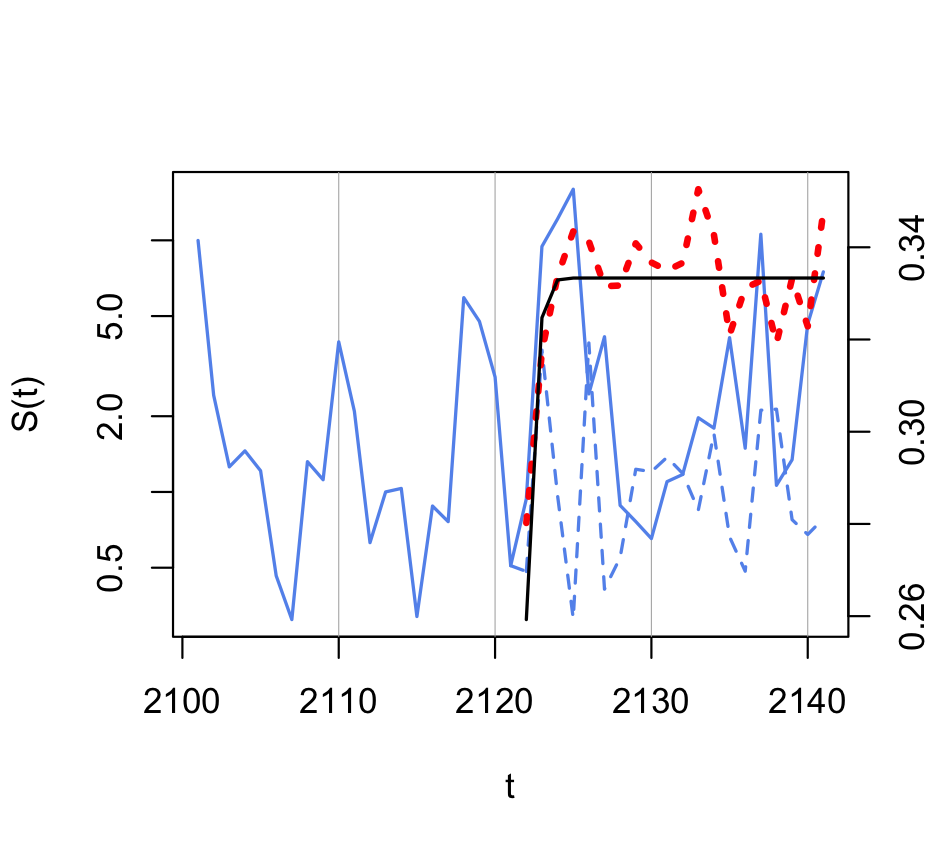}
\end{minipage}

\caption[Prediction of Smith Process]{A 20 time steps' forecast (dashed blue line), together with its corresponding theoretical (black line) and empirical (red line) excursion metric, of a Smith process $S$ (blue line) constructed with a standard deviation of $\sigma_S=1.298$ (left plot) and $\sigma_S=0.482$ (right plot), see \eqref{smith_proc_eq}. For the left plot, $\gamma_{opt}=0$  and for the right plot, $\gamma_{opt}=2$, cf. Table \ref{tab:gam_star}.}
\label{forecast_smith}
\end{figure}
 
\begin{figure}[ht]
\centering
\begin{minipage}[t]{0.48\textwidth}
  \centering
  \includegraphics[width=\linewidth]{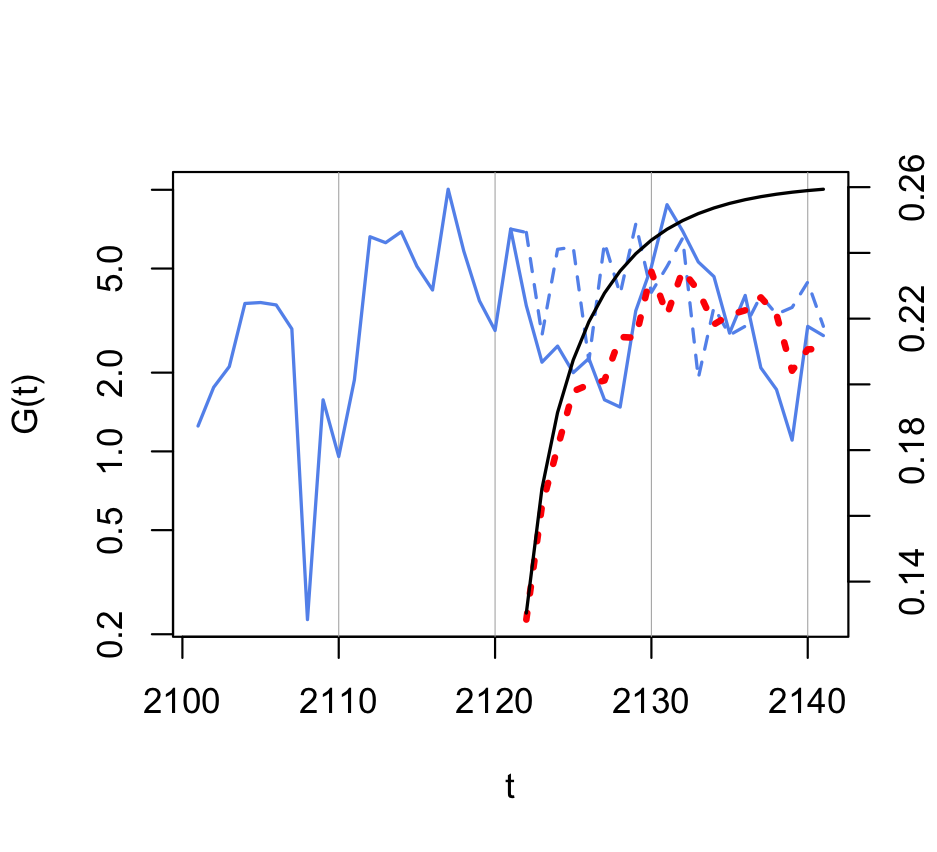}
\end{minipage}\hfill
\begin{minipage}[t]{0.48\textwidth}
  \centering
  \includegraphics[width=\linewidth]{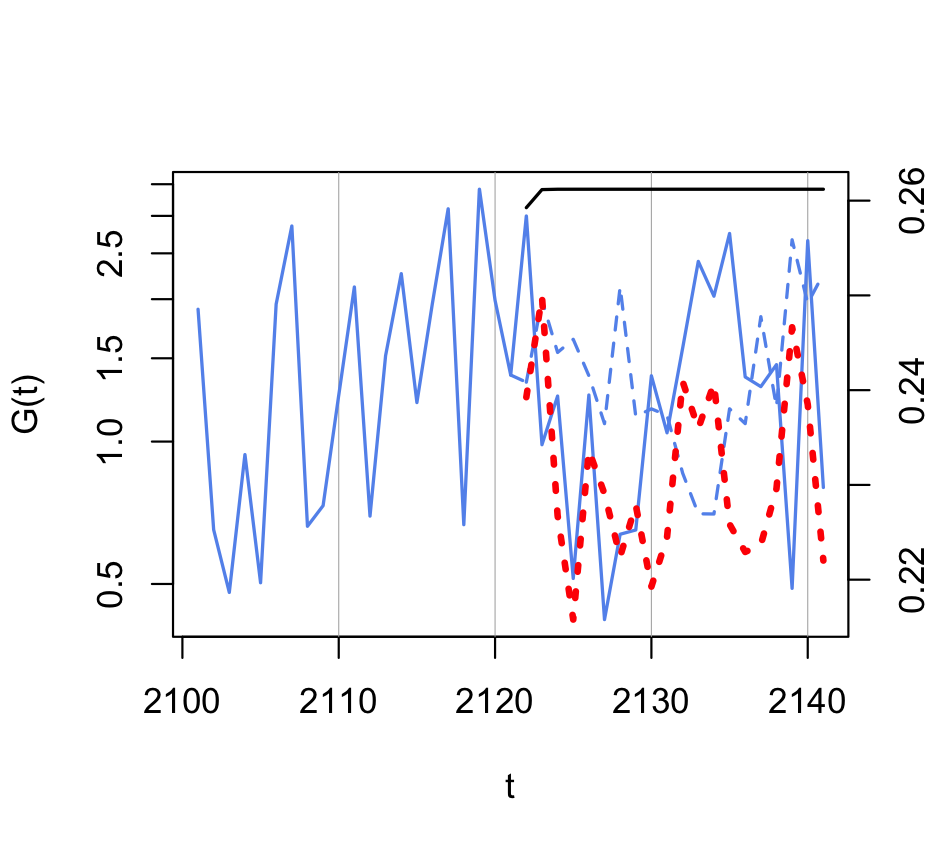}
\end{minipage}

\caption[Prediction of Extremal Gaussian Process]{A 20 time steps' forecast (dashed blue line), together with its corresponding theoretical (black line) and empirical (red line) excursion metric, of an extremal Gaussian process $G$ (blue line). The underlying process $Y$ of $G$ in \eqref{extremalgaussian_proc_eq} is Gaussian with Ornstein--Uhlenbeck covariance function $C(t)= \exp(-\lvert t\rvert /\sigma_G)$, where $\sigma_G=5.039$ (left plot) and $\sigma_G=0.256$ (right plot). For the left plot, $\gamma_{opt}=0$ and for the right plot, $\gamma_{opt}=1$, cf. Table \ref{tab:gam_star}.}
\label{forecast_extremalgaussian}
\end{figure}

\begin{figure}[htbp]
  \centering

  \begin{minipage}[b]{0.32\textwidth}
    \centering
    \includegraphics[width=\textwidth]{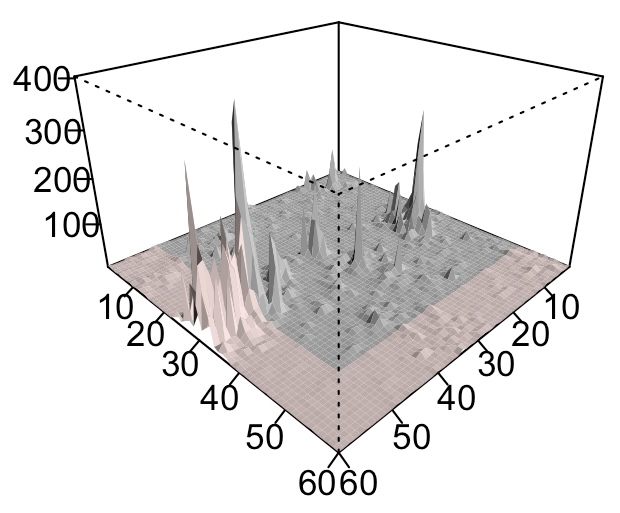} \tiny{\centering{$B$}}
  \end{minipage}\hfill
  \begin{minipage}[b]{0.32\textwidth}
    \centering
    \includegraphics[width=\textwidth]{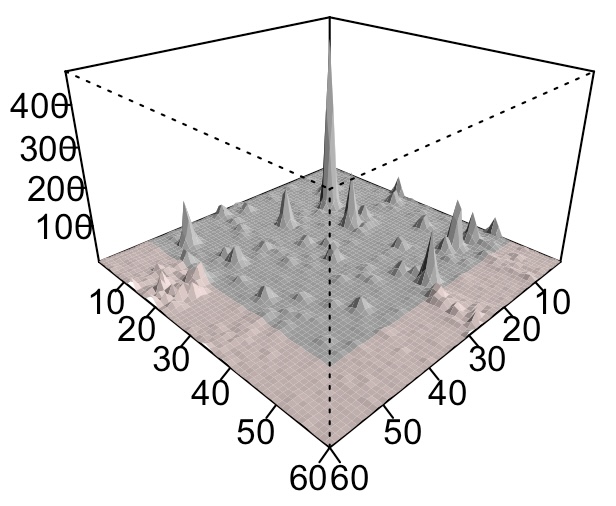}
    \tiny{\centering{$S$}}
  \end{minipage}\hfill
  \begin{minipage}[b]{0.32\textwidth}
    \centering
    \includegraphics[width=\textwidth]{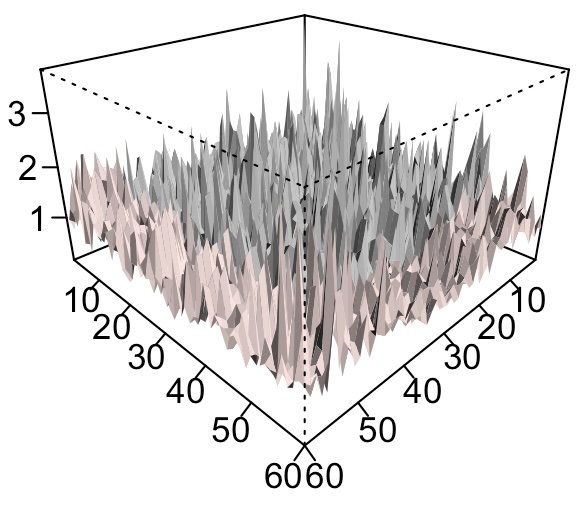}
    \tiny{\centering{$G$}}
  \end{minipage}

  \caption[Prediction of 3D Random Fields]{A ten steps' forecast of random fields $X\in\{B,S,G\}$ in both directions $t_1$ and $t_2$. After observing true values $X_t$  at locations $t \in \{1,\ldots,50\}\times \{1,\ldots,50\}$, the predictor extended the random surface $X$ to $t \in \{1,\ldots,60\}\times \{1,\ldots,60\}$. The Brown-Resnick field $B$ (left) is simulated using $\sigma_B=1.683$, the Smith field $S$ (center) is simulated using $\Sigma=\sigma_S^2I_2$ with $\sigma_S=0.594$ and the extremal Gaussian field $G$ (right) is simulated using $\sigma_G=0.786$. As penalty weights,  the values $\gamma_{opt}$ for $\theta_X=1.6$ from Table \ref{tab:gam_star} are used.}
  \label{forecast_random_field}
\end{figure}




\subsection{Application to annual daily rainfall maxima}\label{subsec72}

The purpose of this section is  to apply the above prediction methodology to forecast the annual amounts of daily rainfall in Munich, Germany. 
The historical rainfall data comes from the National Centers for Environmental Information  \cite{noaa_ncei_website}. We pick the data from the weather station in Munich Nymphenburg because it has the most observations. For some years, the observations are missing. In this case, we fill the missing values by taking the mean of the maximum rainfall data from other weather stations in Munich.
Our goal to predict the maximums for the years 2023-2025 using the observations made in 1879-2022 and compare the forecasts with recently observed values.

Suppose that the annual rainfall maxima $X_1,\ldots,X_M$ observed in some region over $M$ years form a stationary time series. Assuming this, we neglect long term trends such as climate change. In most real life applications,  the true univariate distribution of a stationary time series is not known and needs to be estimated. According to \cite{papalexiou2013}, the Fr\'echet distribution seems to be the best model (out of the three possible extreme value distributions) to fit the time series of annual maxima of daily precipitation. The parameters $(\alpha,\mu,\sigma)$ of the Fr\'echet distribution  $H_{\alpha}(\cdot, \sigma)$ from \eqref{DefG} shifted by $\mu$ can be assessed, e.g. by the maximum likelihood (ML) method \cite{embrechts1997}.
 As random variables $X_j$ are not stochastically independent, we can only calculate quasi ML estimates of the parameters. These are given by
\begin{equation}\notag
	(\hat{\alpha}, \hat{\mu}, \hat{\sigma}) = \!\!\!\! \argmax\limits_{\alpha,\sigma \in\bR^+, \mu\in\bR} \biggl[M \log\biggl(\frac{\alpha}{\sigma}\biggr)  -\sum\limits_{j=1}^M \biggl(\frac{X_j-\mu}{\sigma}\biggr)^{-\alpha} \!\!\!\! - (\alpha+1)\sum\limits_{j=1}^M \log\biggl(\frac{X_j-\mu}{\sigma}\biggr)\biggr].
\end{equation}

The time series of yearly rainfall maxima with its empirical and fitted Fr\'echet  univariate distributions are given in Figure \ref{annual_rainfall_and_fitted_distribution}. As $\hat{\alpha}>2$, the marginal distribution is not heavy-tailed. However, our excursion-based  prediction is applicable here, too. In addition, it has the law-preserving advantage as compared to kriging. The fitted Fr\'echet$(\hat{\alpha},\hat{\mu},\hat{\sigma})$ distribution is used as c.d.f. $F$ in the excursion metric $ \mathcal{E}_{F}$ to forecast the above time series $X=\{ X_t, t\in\bN \}$.

 As a forecast sample, we take the observations $X_t$ with years $t$ ranging from 2019 to 2022. The remaining 140 past observations $X_t$, $t=1879,\ldots, 2018$ are used to form 35 learning samples of size 4. This particular setup has the advantage that we use all 147 observations (3 step prediction using a forecast sample and 35 non-overlapping learning samples of size 3).  For the prediction, the penalty weight $\gamma=2$ was set, since the yearly precipitation maximums are very weakly dependent, compare their empirical autocorrelation function in Figure \ref{rainfall_prediction} (left)    and  Table \ref{tab:gam_star} for $\theta_X=1.9$. The forecast results are shown in Figure \ref{rainfall_prediction} (right). As one can see, the true trajectory (solid blue line) of the annual maximum precipitation time series $X$ lies almost entirely within the forecast envelope (given in yellow), which justifies the quality of the forecasts. 



\begin{figure}[htbp]
    \centering
    \begin{minipage}{0.50\textwidth}
        \centering
        \includegraphics[width=\linewidth]{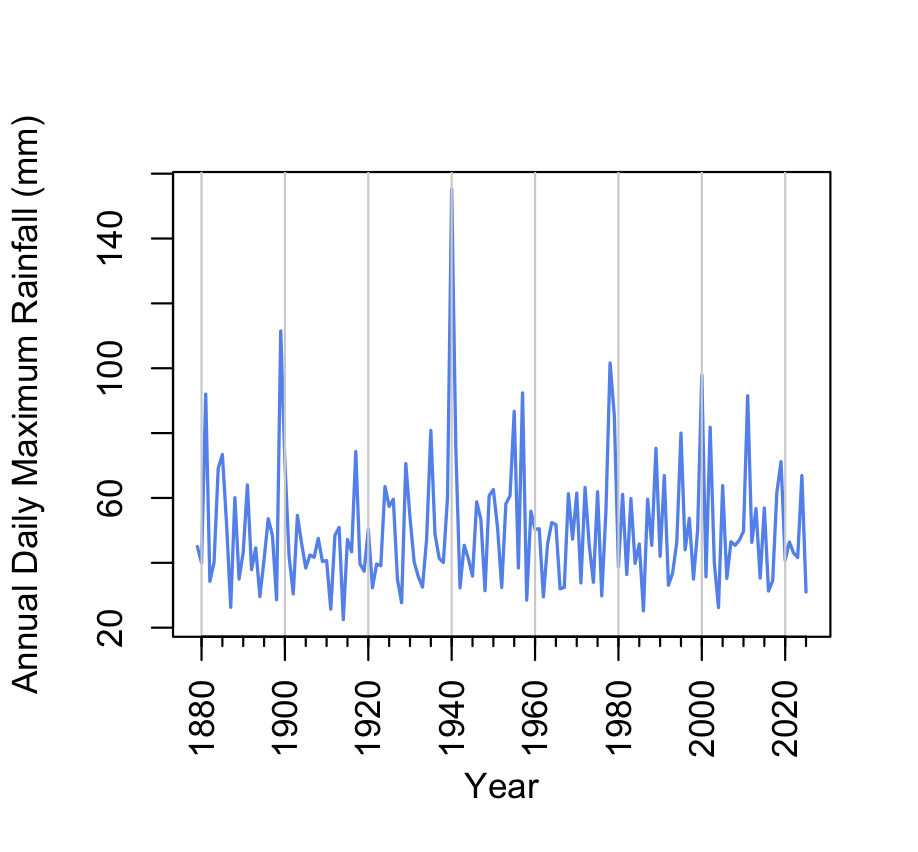}
    \end{minipage}\hfill
    \begin{minipage}{0.50\textwidth}
        \centering
        \includegraphics[width=\linewidth]{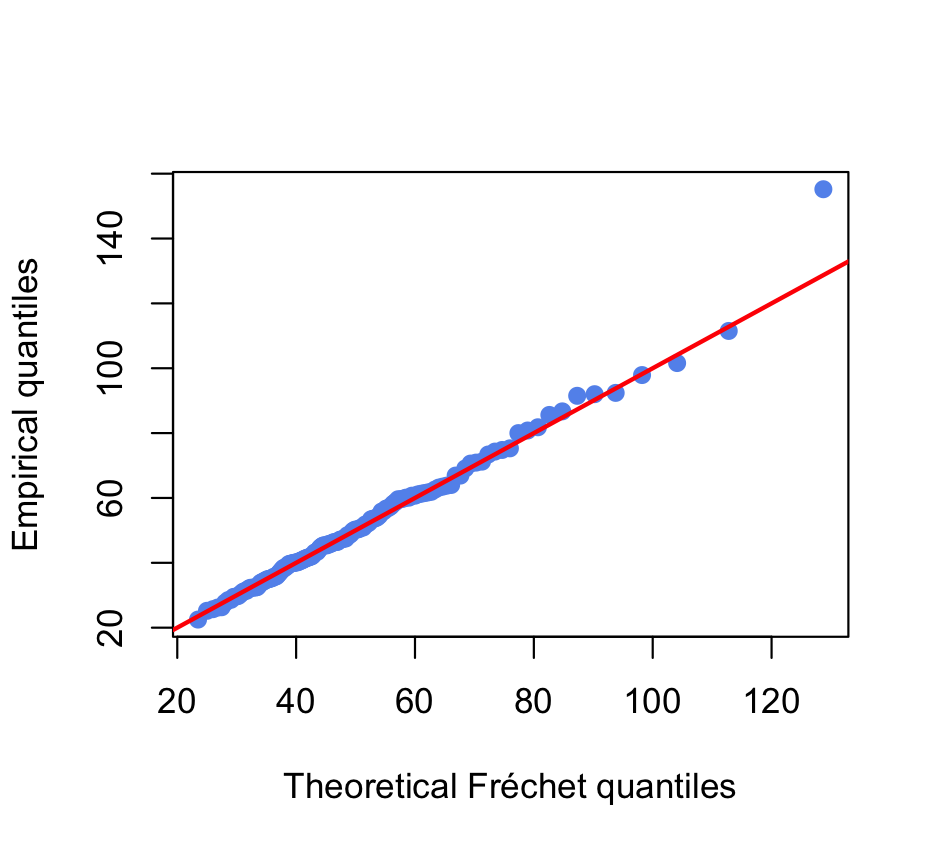}
    \end{minipage}
    \caption{Left:   the yearly maximum of daily rainfall in Munich, Germany from 1879 to 2025. Right: Q-Q plot of the empirical rainfall data against the theoretical Fréchet quantiles with quasi-ML-estimated parameters $\hat{\alpha}\approx7.5263$, $\hat{\mu}\approx-51.4312$ and $\hat{\sigma}\approx92.7826$.}
    \label{annual_rainfall_and_fitted_distribution}
\end{figure}


\begin{figure}[htbp]
    \centering
    \begin{minipage}{0.50\textwidth}
        \centering
      \includegraphics[scale=0.19]{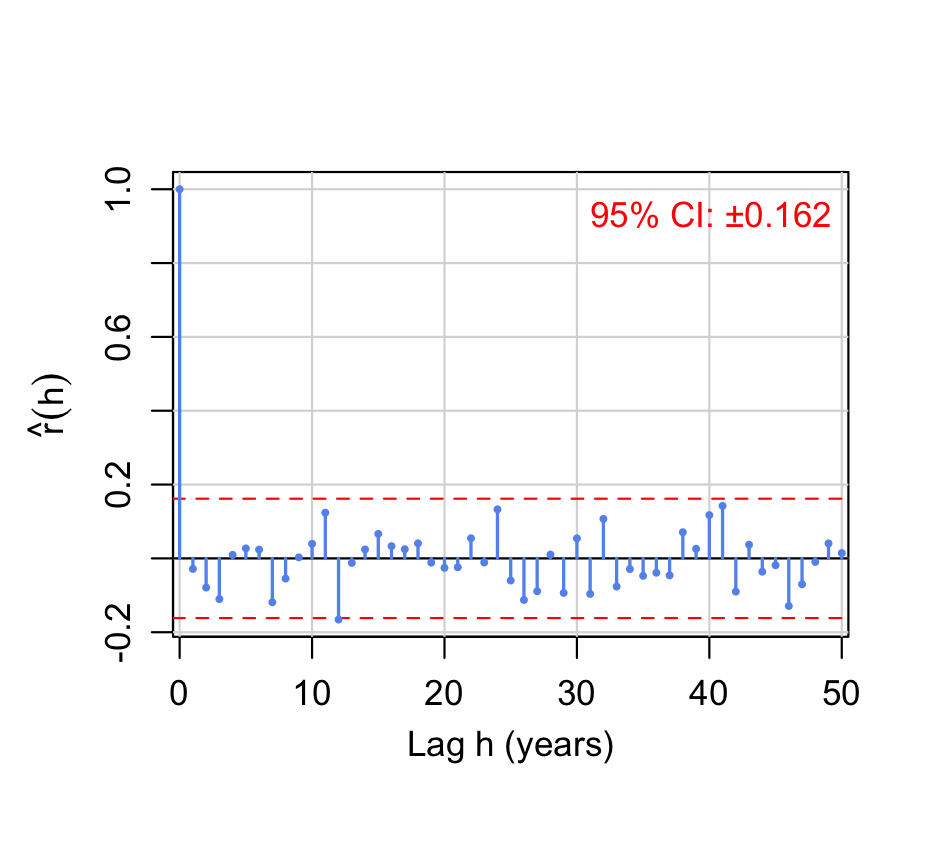}
    \end{minipage}\hfill
    \begin{minipage}{0.50\textwidth}
        \centering
      \includegraphics[scale=0.19]{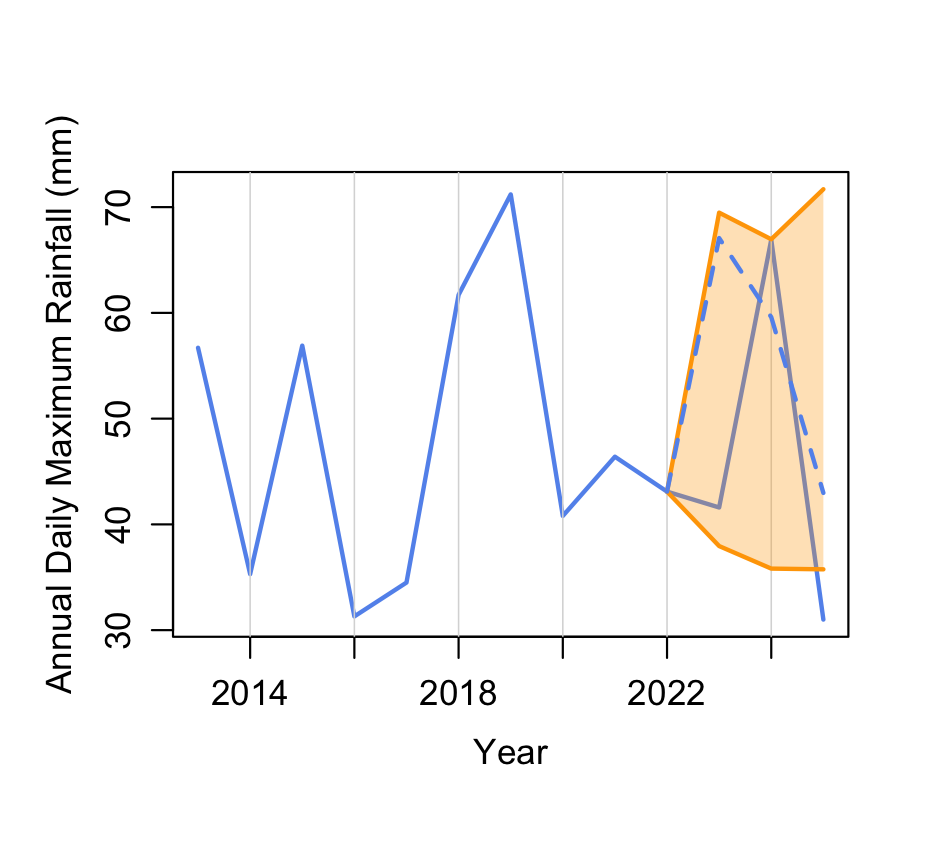}
    \end{minipage}
    \caption{Left: Correlation function for the time series of annual daily rainfall maxima in Munich.  The red dashed lines represent the approximate 95\% confidence region under the null hypothesis of zero autocorrelation. Since most bars are inside this region, there is no strong evidence of significant autocorrelation. Right:  Forecasts for the annual daily rainfall maxima in the years 2023 to 2025. Data of the years 1879-2022 was used in learning samples. The true (observed) time series is shown by the solid blue line. The orange lines mark the maximum and minimum of 100 forecasts using the max-stable predictor \eqref{constraint2} with bootstrap. The dashed blue line yields the forecast using the non-bootstrap formulation \eqref{constraint3}. For every extrapolation, 35 learning samples of size 4 containing data of the years 1879-2018 and a forecast sample containing precipitation data in the years 2019-2022 were used.}
    \label{rainfall_prediction}
\end{figure}


\section{Summary}\label{sec8}
We propose a simple method to predict stationary heavy tailed max-stable random fields. The  advantages lie in its implementational ease, fast computation times and the reliability of its results. Besides, we offer two alternative formulations of the forecast which can be used in different ways. The bootstrapping variant yields randomized forecasts  which can be used to create  confidence  envelopes (bands) that are very likely to include unknown time series values. On the other hand, the formulation without bootstrap produces a unique forecast with lower excursion metrics.

\section*{Acknowledgments}

Dedicated to the memory of Dr. V. Makogin (1987--2024) who unexpectedly passed away on the 8th of May 2024. 


\clearpage
\bibliographystyle{plain}
\bibliography{literature, MakoginLit}

\end{document}